\documentclass{amsart}
\usepackage{amsmath,amssymb,amsthm}
\usepackage{amscd,graphicx}
\usepackage[english]{babel}
\usepackage[backref=page, colorlinks, linkcolor=blue]{hyperref}

\newtheorem{theorem}{Theorem}[section]
\newtheorem{proposition}[theorem]{Proposition}
\newtheorem{lemma}[theorem]{Lemma}
\newtheorem{corollary}[theorem]{Corollary}

\theoremstyle{plain}
\newtheorem{maintheorem}{Theorem}

\theoremstyle{definition}
\newtheorem{definition}[theorem]{Definition}
\newtheorem{remark}[theorem]{Remark}
\newtheorem{example}[theorem]{Example}

\usepackage{stmaryrd}
\usepackage{color}
\usepackage{enumitem}

\usepackage{tikz-cd}

\newcommand{\cA}{{\mathcal A}}

\newcommand{\cC}{{\mathcal C}}

\newcommand{\cH}{{\mathcal H}}

\newcommand{\cW}{{\mathcal W}}

\newcommand{\om}{{\omega}}

\newcommand{\bbN}{{\mathbb N}}

\newcommand{\bbR}{{\mathbb R}}

\begin{document}

\title[Quenched and annealed equilibrium states and applications]{Quenched and annealed equilibrium states for random Ruelle expanding maps and applications}

\author{Manuel Stadlbauer, Paulo Varandas and Xuan Zhang}

\address{Manuel Stadlbauer, Instituto de Matem\'atica, Universidade Federal do Rio de Janeiro\\
 Av. Athos da Silveira Ramos 149, 21941-909 Rio de Janeiro (RJ), Brazil}
\email{manuel@im.ufrj.br}

\address{Paulo Varandas, CMUP and Departamento de Matem\'atica, Universidade Federal da Bahia\\
Av. Ademar de Barros s/n, 40170-110 Salvador, Brazil}
\email{paulo.varandas@ufba.br}

\address{Xuan Zhang, Instituto de Matem\'atica e Estat\'istica, Universidade de S\~ao Paulo\\
R. do Matão, 1010, 05508-090 S\~ao Paulo (SP), Brazil.
}
\email{xuan@ime.usp.br}

\subjclass[2010]{
Primary: 
37A25,   	
37C85,   	
37D35, 
Secondary: 
 37H05,   	
 47B80,   	
 37A50.   	
}

\keywords{Random dynamical systems, quenched and annealed equilibrium states, non-autonomous dynamical systems, decay of correlations, semigroup actions}

\begin{abstract} In this paper we describe the spectral properties of semigroups of  expanding maps acting on Polish spaces,  considering both sequences of transfer operators along infinite compositions of dynamics and integrated transfer operators. 
We prove that there exists a limiting behaviour for such transfer operators, and that these semigroup actions admit 
equilibrium states with exponential decay of correlations and several limit theorems. The reformulation of these results in terms
of quenched and annealed equilibrium states extend results in \cite{Baladi,CRV17}, where the randomness is driven by a 
random walk and the phase space is assumed to be compact. Furthermore, we prove that the quenched equilibrium measures
vary H\"older continuously and that the annealed equilibrium states can be recovered from the latter.   
Finally, we give some applications in the context of weighted non-autonomous iterated function systems, free semigroup actions and 
on the boundary of equilibria.
\end{abstract}

\thanks{
MS was partially supported by CAPES (Programa PROEX da P\'os-Gradua\c c\~ao em Matem\'atica do IM-UFRJ), CNPq (PQ 312632/2018-5, Universal 426814/2016-9). 
PV  was partially supported by CMUP (UID/MAT/00144/2013), which is funded by FCT (Portugal) with national (MEC) and European structural funds through the programs FEDER, under the partnership agreement PT2020, and by 
 Funda\c c\~ao para a Ci\^encia e Tecnologia (FCT) - Portugal, through the grant CEECIND/03721/2017 of the Stimulus of Scientific Employment, Individual Support 2017 Call. XZ was  supported by FAPESP grant \#2018/15088-4. 
}

\maketitle

\tableofcontents
 
\section{Introduction and statement of results}

We consider the joint action of a finite family $\{T_i\}$ of Ruelle expanding maps acting on a complete metric space $X$ from the viewpoint of thermodynamic formalism. This is closely related to the  quenched results in the purely topological context of fibred systems with Ruelle expanding fibres and invertible factor, as studied in  \cite{DenkerGordin:1999,DenkerGordinHeinemann:2002}, the annealed setting in \cite{Baladi} for a random dynamical systems which is modelled by a skew product over 
an ergodic automorphism $\theta: (\Omega,\mathrm P) \to (\Omega,\mathrm P)$ as well as for arbitrary sequences of expanding maps on the unit interval  (\cite{Heinrich:1996,ConzeRaugi:2007}) or general non-autonomous dynamical systems (we refer the reader to \cite{CRV19,Vaienti} and references therein). One of the central questions in this area is effective construction of SRB measures and equilibrium states as it might allow to establish, for example, limit laws or stability under perturbations. Furthermore, recent results, including \cite{CRV17,CRV18}, allow to bridge between the dynamics of fibred systems with the dynamics of semigroup actions, furnishing an important field of applications. 

\smallskip
The previous contexts are naturally related with non-stationary dynamics. On the one hand,
any invariant probability $\mu$ for a fibred system can be disintegrated by 
a measurable family $(\mu_\omega)_{\omega}$ of probabilities, each of which describing typical 
points according to the random orbit
\begin{equation}\label{eq:orbit-random}
T_\om^n:= T_{\theta^{n-1}(\om)} \circ \dots \circ T_{\theta(\om)} \circ T_\om.
\end{equation}
On the other hand, a description of the dynamics of the composition in \eqref{eq:orbit-random} 
for $\mathrm P$-typical points allows to reconstruct the whole random dynamics 
taking the probability $d\mu=d\mu_\om \,d\mathrm P(\om)$.
While many of the known results concern the dynamics of random orbits associated to typical points 
in $\Omega$, there has been recent interest in taking a more embracing approach to describe 
statistical properties in the absence of a reference measure $\mathrm P$, that is in the context of non-autonomous dynamical systems, either using topological or ergodic methods. However, to the best of our knowledge, there are few 
results on the thermodynamic formalism of these non-autonomous dynamical systems, where a major obstruction is caused by the absence of reference probabilities. In a recent result, Atnip et al \cite{Atnip} developed a quenched thermodynamic formalism for
piecewise monotone interval maps satisfying a random covering property (a property on a finite iteration of the transfer operator). 
One of the main contributions of this paper is to provide a general description of the sequential dynamics obtained by composition 
of transfer operators associated to distance expanding maps on Polish spaces, paving the way for a thermodynamic formalism of both 
sequential dynamics, quenched and annealed random dynamical systems.

\smallskip
In the case of annealed random dynamical systems, there are very few results on the existence of equilibrium states and description of their statistical properties, except for the case of the geometric
potential or Bernoulli randomness. In this case the Lebesgue measure is the conformal measure associated to all 
sequential dynamics, the notions of annealed and quenched topological pressure coincide, and the relevant measures are the SRB measures. This situation was considered by Baladi in \cite{Baladi}, where a thermodynamic formalism for annealed random expanding maps driven by 
specific measures on the shift was developed.

In contrast to that work, we overcome the problem of the non-existence of invariant densities due to purely functorial reasons by considering the joint action on the same space and integrate over the possible paths. Such an approach is inspired by some new methods from \cite{BessaStadlbauer:2016,BressaudFernandezGalves:1999,HairerMattingly:2008,KloecknerLopesStadlbauer:2015,Stadlbauer:2015},
and has potential impact in different applications of the thermodynamic formalism, including
a description of invariant measures and equilibrium states for semigroup actions as initiated
in \cite{CRV17,CRV18}.

\smallskip
In what follows we introduce the setting and state the main results of this paper. However, for the sake of simplicity, we postpone several technical definitions to the next sections.   
Throughout, we assume that $(X,d)$ is a complete and separable metric space, and that $T_1, \ldots T_k: X \to X$ are continuous, surjective and Ruelle expanding maps (cf. Definition \ref{def:Ruelle expanding}). Moreover, we always assume that the semigroup  $\mathcal{S}$ generated by these maps is jointly topologically mixing and finitely aperiodic (cf. Definitions \ref{def:jointly topologically mixing} and \ref{def:finitely aperiodic}). 

Moreover, as we are interested in thermodynamic quantities, we fix H\"older continuous functions $\varphi_1, \ldots, \varphi_k: X \to \mathbb{R}$ and define, for a finite word $v=i_1 \ldots i_n$, 
\[ T_v:=  T_{i_n}\circ \cdots \circ T_{i_1} \hbox{ and } \varphi_v := \varphi_{i_1} + \varphi_{i_2} \circ T_{i_1} + \cdots  + \varphi_{i_n} \circ T_{i_1 i_2 \ldots i_{n-1}}. \] 
This then gives rise to a family of Ruelle operators $\{L_v\}$ and a further family of  operators
  $\{\mathbb{P}_{u}^{v}\}$, defined by
\[ L_v(f)(x) := \sum_{T_v(y)=x} e^{\varphi_v(y)}f(y), \quad \mathbb{P}_{u}^{v}(f) = \frac{L_{v}(f \cdot L_{u}(\mathbf{1}))}{L_{uv}(\mathbf{1})}, \]
for $f$ in a suitable function space and with $\mathbf{1}$ referring to the constant function of value $1$. Moreover, in order to guarantee that $L_{v}(\mathbf{1})$ is well-defined, we also assume that the functions $\varphi_i$ are summable (cf. Definition \ref{def:hoelder-summable}).
 The two main features of these quotients are that $\mathbb{P}_{u}^{v}(\mathbf{1}) = \mathbf{1}$ and that the iteration rule $\mathbb{P}_{uv}^w\circ \mathbb{P}_{u}^v = \mathbb{P}_{u}^{vw}$ holds. It follows from the first that the dual operators $\{(\mathbb{P}_{u}^{v})^\ast\}$ act on the space of probability measures  $\mathcal{M}_1(X)$, and from the second that a certain contraction with respect to one of those operators implies geometric convergence of the family. Our first principal result now establishes this kind of convergence. In here, $\overline{W}$ refers to the Vaserstein metric and  $\overline{D}$ to the H\"older coefficient with respect to the equivalent metric $d^\ast$ (cf. \eqref{eq:d star}). We refer the reader to Section~\ref{sec:quotients} for the necessary definitions
 and notations.    

\begin{maintheorem}\label{thm:mainA}
Suppose  the Ruelle-expanding semigroup $\mathcal{S}$ is jointly topologically mixing and finitely aperiodic, and that every potential $\varphi_i$ is $\alpha$-H\"older and summable.
 Then there exist $k_0 \in \bbN$ and $s \in (0,1)$ such that for all finite words $u,v$ with $|v|\ge k_0$ and $\nu_1 , \nu_2\in \mathcal{M}_1(X)$ and every H\"older continuous observable $f: X\to\mathbb R$ with  $\overline{D}(f) < \infty$,
\begin{align*}
\overline{W}({\mathbb{P}_{u}^{v}}^\ast(\nu_1), {\mathbb{P}_{u}^{v}}^\ast (\nu_2)) &\leq  s^{|v|} \overline{W}( \nu_1 , \nu_2),\\
\quad \overline{D}(\mathbb{P}_{u}^{v}(f))  &\leq s^{|v|} \overline{D}(f).
\end{align*}
\end{maintheorem}

This theorem implies that, for {any} infinite word $\omega =i_1 i_2 \ldots$ and measure $\nu\in \mathcal{M}_1(X)$, the limit
\[\mu_{\omega}:=  \lim_{l \to \infty} \left( {\mathbb{P}_{\emptyset}^{i_1 \ldots i_l}}\right)^\ast(\nu)\]
exists, is independent of $\nu$ and the speed of convergence is exponential. This means that, under some mild assumptions on the set of Ruelle expanding maps, any non-autonomous sequence of dynamics admits a probability measure that rules its dynamics and that this measure is a non-autonomous conformal measure in the following sense: 
there exists $\lambda_{u,\omega}> 0$ such that $L_u^\ast(\mu_\omega) = \lambda_{u,\omega} \mu_{u\omega}$ (see Proposition~\ref{prop:conformal}).  
Furthermore, for any left infinite word $\tilde\omega = \ldots i_{-2} i_{-1}$, the limit
\[\mu_{\tilde{\omega},\omega}:=  \lim_{l \to \infty} \left( {\mathbb{P}_{i_{-l} \ldots i_{-1}}^{i_1 \ldots i_l}}\right)^\ast(\nu)\]
exists, varies H\"older continuously with $\omega$, is independent of $\nu$, and the speed of convergence is exponential. As shown in  Proposition~\ref{prop:equilibrium}, this measure is invariant in the non-autonomous setting, and if 
$\tilde{\omega}$ and $\omega$ are periodic extensions of the finite word $w$, that is $\tilde{\omega} = \ldots ww$ and ${\omega} =  ww \ldots$, then  $\mu_{\tilde{\omega},\omega}$ is the unique equilibrium state of $(T_w,\varphi_w)$ (cf. Proposition~\ref{prop:eq state}). In fact, the set of all measures  $\{\mu_{\tilde{\omega},\omega}\}$, where $\tilde{\omega}$, $\omega$ run through all infinite words is the closure of these equilibrium states and can be used to define a compactification of the semigroup (Proposition~\ref{prop:boundary}).  

A further application of Theorem \ref{thm:mainA} is related to an invariance principle as the contraction allows to apply the general invariance principle in \cite{CunyMerlevede:2015} and gives rise to the following result (for a similar result for continued fractions with restricted entries, see \cite{StadlbauerZhang:2017a}). $[\omega]_n$ stands for the initial $n$-word of an infinite word $\omega$.

\begin{maintheorem}\label{thm:mainD}
Suppose the finitely Ruelle-expanding semigroup $\mathcal{S}$ is jointly topologically mixing and finitely aperiodic, and that every potential $\varphi_i$ is $\alpha$-H\"older and summable. 
Suppose $\omega\in\Sigma$, $f\in\cH_\alpha$. Let $f_n=f -\int f\circ T_{[\omega]_n} d\mu_{\omega}$ for every $n\in\mathbb N_0$ and let $s_n^2 = \mathbb E_{\mu_\omega}(\sum_{k=0}^{n-1} f_k\circ T_{[\omega]_k})^2$ for $n\ge 1$ and assume that 
$ \sum_n s_n^{-4}<\infty$. Then there exists a sequence $(Z_n)$ of independent centred Gaussian random variables such that 
\begin{gather*}
\sup_n\left|\sqrt{\textstyle \sum_{k=0}^{n-1} \mathbb E_{\mu_\omega} Z_k^2}-s_n\right|<\infty,\\
\sup_{0\leq k \leq n-1} \left| \textstyle\sum_{i=0}^k f_i\circ T_{[\omega]_i} - \sum_{i=0}^k Z_i \right| = o(\sqrt{s^2_n \log \log s^2_n}) \hbox{ a.s.}
\end{gather*}
\end{maintheorem}

\smallskip

We then relate and apply these results to random dynamical systems, that is we assume that the $T_i$ are chosen with respect to a given probability measure $\rho$. So, it is sufficient to fix a measure $\rho$ either on the shift spaces $\Sigma := \{1,\ldots,k\}^\mathbb{N}$ or $\Sigma_\mathbb{Z} := \{1,\ldots,k\}^\mathbb{Z}$ and consider the almost sure behaviour, referred to as \emph{quenched}, and the behaviour in average, referred to as \emph{annealed} behaviour. In this setting, Proposition~\ref{prop:conformal} provides 
existence and exponential decay towards the quenched random conformal measure $\mu_\omega$, whereas the bilateral result in Proposition~\ref{prop:equilibrium} implies the same statement for the quenched equilibrium state $\mu_{\tilde{\omega},\omega}$. 

In order to relate these quenched results to their annealed counterparts, we consider in here as in \cite{Baladi} also the annealed operators 
 \[  \mathcal{A}_n := \sum_{|w|=n} \rho(\{\omega: [\omega]_n=w\}) L_w .\]
A fundamental problem of these operators is that, in general, 
$\mathcal{A}_{n+m} \neq \mathcal{A}_n \circ \mathcal{A}_m$, which makes it impossible to apply methods from spectral theory. However, if we assume that $\rho$ is supported on a topologically mixing, one-sided subshift of finite type, it is possible to control the asymptotic behaviour of $\{\mathcal{A}_n\}$, which is our second main result. In here, $\theta$ refers to the one-sided shift map. 

\begin{maintheorem} \label{thm:mainB}
Suppose the Ruelle-expanding semigroup $\mathcal{S}$ is jointly topologically mixing and finitely aperiodic, and that every potential $\varphi_i$ is $\alpha$-H\"older and summable. Moreover, suppose that $\rho$ is supported on a topologically mixing, one-sided subshift of finite type and that
$d\rho/d\rho\circ \theta$ is H\"older continuous. Then there exist  $r\in(0,1)$, a positive function $h\in\cH_\alpha$ and $\beta>0$ such that for all $f \in \mathcal{H}_\alpha$ and every large $n \ge 1$,
\[  \left| \frac{\mathcal{A}_n(f)(x)}{\beta^n h(x)}  - \int f d\pi \right| \ll r^n (\overline D(f) +\|f\|_m).\]
 \end{maintheorem} 

Now assume that $\rho$ is a Bernoulli measure, that is the maps $T_i$ are chosen independently. Then, by independence, it follows that $\mathcal{A}_n = (\mathcal{A}_1)^n$. Hence, as an immediate corollary, one obtains that 
\[ (\mathcal{A}_1)^n (hf)(x) /\beta^n h(x) \longrightarrow \int f(x) h(x) d\pi(x)\]
exponentially fast, which is a well-known version of Ruelle's operator theorem for independently chosen maps $T_i$ (cf. Proposition~3.1 in \cite{Baladi}). As this is the key step for existence and uniqueness of the annealed equilibrium state (cf. Proposition~3.3 in \cite{Baladi}), one obtains  Theorem 1 in \cite{Baladi} for i.i.d. Ruelle expanding maps  as a corollary.

\smallskip
We now return to the general case of a one-sided subshift of finite type with exponential decay of correlations and now assume, in addition, that $\rho$ is $\theta$-invariant. In this setting, we obtain an annealed version of  decay of correlations. 

\begin{maintheorem}\label{thm:mainC}
Suppose that the assumptions of Theorem~\ref{thm:mainB} hold and that $\rho$ is $\theta$-invariant.  Then there exist a probability measure $\tilde{\pi}$, $r \in (0,1)$ and $k_1 \in \mathbb{N}$ such that
\begin{align*} 
&\left| \int \sum_{|v|=n} \mathbf{1}_{[v]}(\omega) f (T_v(x)) g(x) d\mu_\omega(x) d\rho(\omega) - 
\int f d\tilde{\pi} \int g d\mu_\omega d\rho \right| \\
& \leq   r^n \int |f| d\mu_\omega d\rho \left(  \overline D(g) + \int |g| d\mu_\omega d\rho \right)
\end{align*}
for all $g \in \cH_\alpha$ and $f: X \to \mathbb{R}$ integrable with respect to $d\mu_\omega(x) d\rho(\omega)$.
\end{maintheorem}

The latter reveals an unexpected connection between quenched and annealed dynamics. Indeed, it is noticeable that despite the fact that quenched and annealed random dynamical systems often measure different complexities of the dynamics (see e.g. \cite[Proposition~8.3]{CRV17} for an explicit formula in the context of free semigroup actions), in Theorem~\ref{thm:mainC} we obtain an annealed decay of correlations with respect to a
probability $d\mu_\omega d\rho$ obtained via quenched asymptotics.
These results for both quenched and annealed dynamical systems will appear as Theorems~\ref{theorem:contraction}, \ref{theo:annealed-conformal}, \ref{theo:annealed decay of correlations} and \ref{thm:asip} below. Moreover, the authors would like to point out, that according to their knowledge, Theorems~\ref{thm:mainB} and \ref{thm:mainC} are the first annealed results for a dependent choice of the maps $\{T_i\}$. Finally, in Section~\ref{sec:app}, we discuss  applications to non-autonomous conformal iterated function systems, the thermodynamic formalism of semigroup actions and a boundary construction through equilibrium states.

\section{Semigroups of Ruelle expanding maps on non-compact spaces} 
We always assume that $(X,d)$ is a complete and separable metric space and 
that $\cW$ is a finite alphabet. For every $i\in \cW$ let $T_i:X \to X$ be a continuous, surjective transformation and let $\mathcal{S}$ be the semigroup generated by $\{T_i\}_{i\in \cW}$, i.e.
$$\mathcal{S} =\{ T_{i_k}\circ T_{i_{k-1}} \circ \cdots \circ T_{i_1} :  k \in \bbN, \, {i_1},{i_2},\ldots, {i_k} \in \cW \}.$$
For every $k\in\mathbb N$ 
and every finite word $v = {i_1}{i_2}\ldots {i_k}\in \cW^k$, set
\[ T_v:=  T_{i_k}\circ \cdots \circ T_{i_1}. \]
Then each element of $\mathcal{S}$ is equal to  $T_v$ for some finite word $v$, but $v$ might not be uniquely determined (e.g. if two generators $T_a, T_b$ commute then $T_{ab}=T_{ba}$). 
Observe that, with the usual concatenation of words, we have that $T_{vw} = T_w \circ T_v$ and, in particular, that the map from $\bigcup_{k\geq 1} \cW^k \to \mathcal S$ given by $v \mapsto T_v$ is a semigroup anti-homomorphism, referred to as the coding of $\mathcal{S}$. With this coding, it naturally defines a free semigroup action $\mathcal S \times X \to X$, $(T_v,x) \mapsto T_v(x)$ determined by $\mathcal S$.

For every finite word $v\in\cW^k$, denote its length by $|v|=k$.  For $x\in X$ and $A\subset X$, let $B_r(x)=\{y\in X: d(x,y)<r\}$ and $\ B_r(A)=\{y\in X: d(x, y)<r \text{ for some } x\in A\}.$ For a finite word $v=i_1\ldots i_k$, define dynamical distance
$$d_v(x,y) :=\sup\{d(x,y),\, d(T_{i_1\ldots i_j}(x), T_{i_1\ldots i_j}(y)), 1\le j<|v| \}$$
and dynamical ball 
$$B_r^v(x):=\{y\in X: d_v(x,y)<r\}.$$ 

Later on we will also consider infinite words. 
The transformations $T_i, i\in\cW$ in this paper are always \emph{Ruelle-expanding} maps as introduced in (\cite{Ruelle2004}). However, in here, we do not require that the base space is compact and, in particular, the set of preimages of a point might be countably infinite. Recall that this notion of expanding map  is defined as follows.
\begin{definition}
$T$ is said to be \emph{$(a,\lambda)$-Ruelle-expanding}, for some $a>0$ and  $\lambda \in (0, 1)$, if for any $x, {y}, \tilde{x} \in X$  with $d(x, {y})<a$ and $T(\tilde{x})=x$, there exists a unique $\tilde{y}\in X$  with $T(\tilde{y})={y}$ and $d(\tilde{x}, \tilde{y})<a$, and such that this $\tilde y$ satisfies 
\[ d(\tilde{x}, \tilde{y}) \leq  \lambda d(x,y) . \]
\end{definition}
Examples of Ruelle-expanding maps include $C^1$-expanding maps on compact Riemannian manifolds, distance expanding maps on compact metric spaces and one-sided subshifts of countable type. In particular our context includes distance expanding maps on non-compact metric spaces. Observe that, as we only consider a finite alphabet $\cW$, we may choose the same parameters $a$ and $\lambda$ for all $T_i, i\in\cW$. 
\begin{definition} \label{def:Ruelle expanding}
The semigroup $\mathcal S$ generated by $\{T_i\}_{i\in\cW}$ is said to be a \emph{$(a,\lambda)$-Ruelle-expanding semigroup} if every $T_i, i\in\cW$, is $(a,\lambda)$-Ruelle-expanding.
\end{definition}

{We extend to the semigroup $\mathcal S$ the notions of topological mixing and finite aperiodicity, which are usually defined for the iteration of a single map. They are known from graph directed Markov systems (\cite{MauldinUrbanski:2003}) or from the  b.i.p.-property for shift spaces (\cite{Sarig:2003a}).

\begin{definition}\label{def:jointly topologically mixing}
$\mathcal S$ is said to be \emph{jointly topologically mixing} if, for all open sets $U,V \subset X$, there exists $m \in \mathbb{N}$ such that $g_{w}^{-1}(U) \cap V \neq \emptyset$ for all finite words $w$ with $|w|\geq m$. 
\end{definition}

\begin{definition}\label{def:finitely aperiodic}
A $(a,\lambda)$-Ruelle-expanding semigroup $\mathcal S$ is said to be ($n$-)\emph{finitely aperiodic} (see Figure \ref{fig:bip}) if there exist $n\in\mathbb N$,  a finite subset $K \subset X$ and  $r>0$ such that for all $x \in X$ and $w \in \cW^n$ one can find $\xi, \eta\in K$ satisfying
\begin{enumerate}[label=(\roman*)] 
\item there is $\xi^\ast \in T_w^{-1}(\xi)$ with $d_w(x, \xi^\ast)<a$,
\item there is $x^\ast\in T_w^{-1}(x)$ with $d(x^\ast,\eta)< a$ and $d_w(x^\ast, \eta)<r$.
\end{enumerate}
\end{definition}
The first condition is modelled after the big image condition, the second after the big preimage condition.

\begin{figure}[htbp] 
   \centering
   \includegraphics[width=\textwidth]{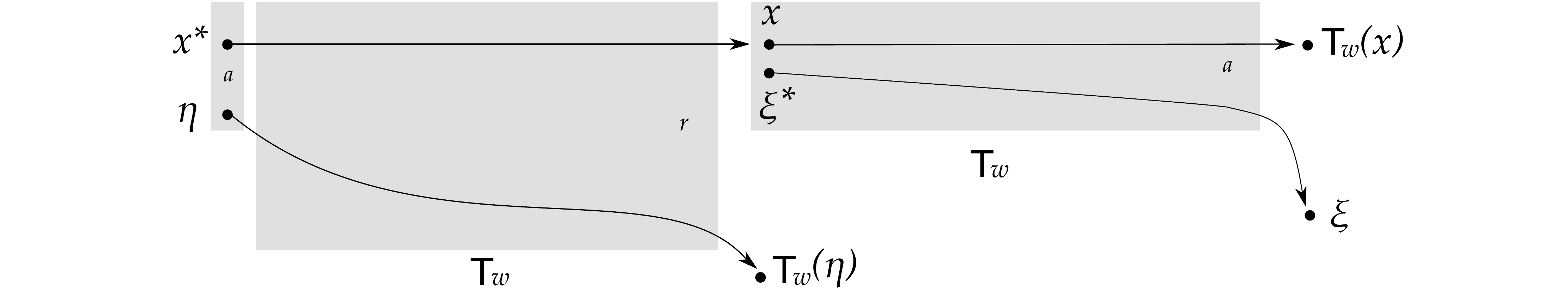}     \caption{Finite aperiodicity}
   \label{fig:bip}
\end{figure} 

\begin{remark}
Any Ruelle-expanding semigroup defined on a compact space $X$ is $n$-finitely aperiodic for every $n\in\mathbb N$,  which can be seen by the following argument. Let $K$ be a finite set such that $X\subset \cup_{z\in K} B_{a/2}(z)$ and let $r=\rm{diam} (X)$. Choose $\xi\in K\cap B_a(T_w(x))$, then the Ruelle expanding property assures the existence of $\xi^\ast$ and hence condition (i). Choose any $x^\ast\in T_w^{-1}(x)$ and  $\eta\in K\cap B_a(x^\ast)$, then condition (ii) follows.
\end{remark}

Without specifying, $\mathcal S$ is always $(a,\lambda)$-Ruelle-expanding in this paper. We use the notations $x\ll y, x\gg y, x\asymp y$ to indicate that there exists a  positive constant $C$ such that $x\le Cy, x\ge Cy, C^{-1}y\le x\le C y$ respectively.

\section{Quotients of Ruelle operators}\label{sec:quotients} 
In this section we introduce a family of quotients of Ruelle operators, which will act as strict contractions on the set of probability measures. It provides an effective construction of
the relevant measures, whereas a normalisation of the Ruelle operators through invariant functions has no dynamical significance in the setting of semigroups or sequential dynamics due to purely functorial reasons, as noted in  Remark \ref{remark:equi} below.

To begin with, let $\varphi_i:X \to \bbR$, $i\in\cW$ be a continuous function. We also call $\varphi$ a potential. Define for a finite word $v = {i_1} {i_2} \ldots {i_k}\in \cW^k$ 
\[ \varphi_v(x):= \varphi_{{i_1}}(x) + \varphi_{{i_2}}(T_{i_1}(x)) + \cdots + \varphi_{{i_k}}(T_{i_1\ldots i_{k-1}}(x)).\] 
Then the Ruelle operator $L_v$ is defined by
\[ L_v(f)(x) := \sum_{T_v(y)=x} e^{\varphi_v(y)}f(y)\]
for $f$ in a suitable function space. Note that it follows from $T_v\circ T_u = T_{uv}$ that  
 $L_v\circ L_{u}= L_{uv}$ for any two finite words $u, v$. 
We now define the adequate function space. For $\alpha \in (0,1]$ and $f:X \to \bbR$,  the H\"older coefficient $D_\alpha(f)$  is
$$D_\alpha(f):= \sup_{x,y \in X, x \neq y} \frac{|f(x)-f(y)|}{d(x,y)^{\alpha}}$$
 and the space  of $\alpha$-H\"older functions $\cH_\alpha^*$ is 
$$
\mathcal{H}^\ast_\alpha := \left\{ f :  D_\alpha(f) < \infty \right\}.
$$
Let $\cH_\alpha$ denote  the subspace of bounded functions in $\mathcal{H}^\ast_\alpha$. It is well known that $\cH_\alpha$ is a Banach space with respect to the norm $\|\cdot\|:=\|\cdot\|_\infty + D_\alpha(\cdot)$. We are now in position to specify the class of potentials considered in here.
\begin{definition} \label{def:hoelder-summable}
We refer to $\varphi_i$ as a \emph{$\alpha$-H\"older potential}
if $\varphi_i \in \mathcal{H}^\ast_\alpha$. Moreover, for any finite word $v$, we say that $\varphi_v$ is a \emph{summable potential} if  
$\|L_v(\mathbf{1})\|_\infty < \infty$. 
\end{definition}
Suppose $\varphi_i$ is $\alpha$-H\"older for every $i\in \cW$. We shall estimate distortion of $\varphi_v$.  Due to $(a,\lambda)$-Ruelle-expanding property, for $v=i_1\ldots i_k \in \cW^k$ and $x, y, \tilde{x}\in X$ with $d(x, y) < a$ and $T_v(\tilde x)=x$, there exists a unique point $\tilde y\in T_v^{-1}(y)\cap B_a^v(\tilde x)$. 
Moreover
$$d(\tilde x, \tilde y) < \lambda^k d(x, y),\quad  d(T_{i_1 \ldots i_j}(\tilde x), T_{i_1 \ldots i_j}(\tilde y)) < \lambda^{k-j}  d({x},{y}),\  1\le j<k.$$ Hence, the inverse branch 
\begin{equation}\label{inverse branch}
(T_v)_{\tilde x}^{-1}: B_a(x) \to  B_a^v(\tilde x),\quad  y\mapsto \tilde y
\end{equation}
is well defined and contracts the distance at every intermediate step by $\lambda$. 
It follows that, for any pair $x, y$ with $d(x,y)<a$, there is a bijection from $T_v^{-1}(x)$ to $T_v^{-1}(y)$ given by 
\begin{equation}\label{bijection of inverse}
 \tilde x\mapsto \tilde y_{\tilde x}:=(T_v)^{-1}_{\tilde x}(y). 
 \end{equation}
Now H\"older continuity implies that whenever $d(x,y)<a$,
\begin{equation}\label{eq:bounded_distortion}
|\varphi_v(\tilde x)-\varphi_v(\tilde y_{\tilde x})| \leq \frac{\max_{i \in \cW} D_\alpha(\varphi_{i})}{1-\lambda^\alpha} d({x},{y})^{\alpha} =: C_\varphi d({x},{y})^{\alpha}.
\end{equation}
It follows from a simple argument that  $L_v$ maps $\mathcal{H}_\alpha$ to $\cH_\alpha$ if $\varphi_v$ is also summable.

As we are interested in operators who leave invariant the constant function $\mathbf{1}$, define for finite words $u, v$ 
\begin{eqnarray*}
\mathbb P_{u}^v(f) &:=& \frac{L_v(f  \cdot L_{u}(\mathbf{1}))}{L_{uv}(\mathbf{1})} =  \frac{L_{uv}(f\circ T_u)}{L_{uv}(\mathbf{1})}.
\end{eqnarray*}
It is clear from the definition that 
$$
\mathbb{P}_{u}^v(\mathbf{1}) =\mathbf{1}. 
$$
The motivation to consider these families of operators stems from the simple observation that, 
for finite words $u,v,w$, 
$$
\mathbb{P}_{uv}^w\circ \mathbb{P}_{u}^v(f) = \frac{L_w( \mathbb{P}_{u}^v(f)  \cdot 	L_{uv}(\mathbf{1}))}{L_{uvw}(\mathbf{1})} 
=  \frac{L_w( L_v(f \cdot L_u(\mathbf{1})))}{L_{uvw}(\mathbf{1})} =  \mathbb{P}_{u}^{vw}(f). 
$$
Hence with  $$\mathbb{P}^{w}(f) := L_w(f)/L_w(\mathbf{1}),$$ for a sequence of finite words $v_1, \ldots v_k$, 
\begin{eqnarray} \label{eq:iteration_1} \mathbb{P}^{v_1\ldots v_k} = 
 \mathbb{P}_{v_1\ldots v_{k-1}}^{v_k} \circ   \mathbb{P}_{v_1\ldots v_{k-2}}^{v_{k-1}} \circ \cdots \circ 
 \mathbb{P}_{v_1 v_2}^{v_3} \circ  \mathbb{P}_{v_1}^{v_2} \circ \mathbb{P}^{v_1}.
\end{eqnarray}
As a first result, we obtain  
$\mathcal{H}_\alpha$-invariance of these quenched operators. 
\begin{lemma}\label{lem:doeblin-fortet}  $\mathbb{P}_{u}^{v}$
is a bounded operator on $\mathcal{H}_\alpha$. Furthermore, for 
 $f \in \mathcal{H}_\alpha$ and $x,y$ with $d(x,y)<a$
\begin{align}
\label{eq:doeblin-fortet-sequencial} |\mathbb{P}_{u}^{v}(f)(x) - \mathbb{P}_{u}^{v}(f)(y)| \leq C_\varphi \left(2 \|f\|_\infty + \lambda^{|v|} D_{\alpha}(f) \right) d(x,y)^\alpha.
\end{align}
\end{lemma}

\begin{proof} Following in verbatim the proof of Lemma 2.1 in \cite{BessaStadlbauer:2016}, one obtains that, for  $x,y$ with $d(x,y)<a$, 
\[ |L_v( f L_u(\mathbf{1}))(x) -  L_v( f L_u(\mathbf{1}))(y)|
\leq C_\varphi L_{uv}(\mathbf{1})(x) (\|f\|_\infty + \lambda^{|v|}D_\alpha(f)) d(x,y)^{\alpha}. \] 
The estimate \eqref{eq:doeblin-fortet-sequencial} 
follows from this as in \cite{BessaStadlbauer:2016}. It remains to show that the operators are bounded and leave invariant $\mathcal{H}_\alpha$. As $\mathbb{P}_{u}^{v}$ maps positive functions to positive functions
and $\mathbb{P}_{u}^{v}(\mathbf{1})= \mathbf{1}$, we have $\|\mathbb{P}_{u}^{v}(f) \|_\infty \leq \|f\|_\infty$. 
Furthermore, by considering the cases $d(x,y)<a$ and $d(x,y)\geq a$ separately, we obtain
\[ D_\alpha(\mathbb{P}_{u}^{v}(f)) \leq 
 \max\left\{ C_\varphi \left(2 \|f\|_\infty + \lambda^{|v|} D_{\alpha}(f)\right), 2 a^{-\alpha} \|f\|_\infty \right\}, \]
which proves that  $\mathbb{P}_{u}^{v}:\mathcal{H}_\alpha \to \mathcal{H}_\alpha $ is a well-defined and bounded operator. 
\end{proof}

We observe that Lemma \ref{lem:doeblin-fortet}, which requires H\"older continuity of the potentials and no further assumption 
on topological irreducibility, is one of the principal ingredients to prove
that the duals of the previous operators act as contractions on the space of probabilities.
The other ingredient is the following result for which  finite aperiodicity is essential.

\begin{lemma} \label{comparability}
 Suppose that $\mathcal{S}$ is jointly topologically mixing and finitely aperiodic, and that every 
 $\varphi_i$ is $\alpha$-H\"older and summable.  Then  $L_v(\mathbf{1})(x) \asymp L_v(\mathbf{1})(y)$, that is, there exists $C>0$ such that $1/C<L_v(\mathbf{1})(x)/  L_v(\mathbf{1})(y)< C$ for all finite words $v$ and $x, y\in X$.
\end{lemma}

\begin{proof} First note that for any $x,y\in X$ with $d(x,y)<a$ and any finite word $v$, the bijection \eqref{bijection of inverse} and the estimate \eqref{eq:bounded_distortion} imply that
$L_v(\mathbf{1})(x) \asymp L_v(\mathbf{1})(y)$. 

Suppose $\mathcal S$ is $n$-finitely aperiodic. Let $K$ a finite set and $r>0$ be given by finite aperiodicity. It follows from Ruelle-expanding and jointly topological mixing that there exists $m\in \mathbb N$ such that for all $\xi, \eta \in K$ and $|w|\ge m$, there exists $\eta^* \in X$ with $T_w(\eta^*)=\eta$ and $d(\eta^*, \xi)<a$.

We now show the lemma for any $x, y\in X$ and all finite words $v$ with $|v|>2n+m$. Take such a finite word $v$, we will select preimages of $x$ as follows, illustrated in Figure~\ref{fig:comparable-construction}.
\begin{figure}[h] 
   \centering
   \includegraphics[width=\textwidth]{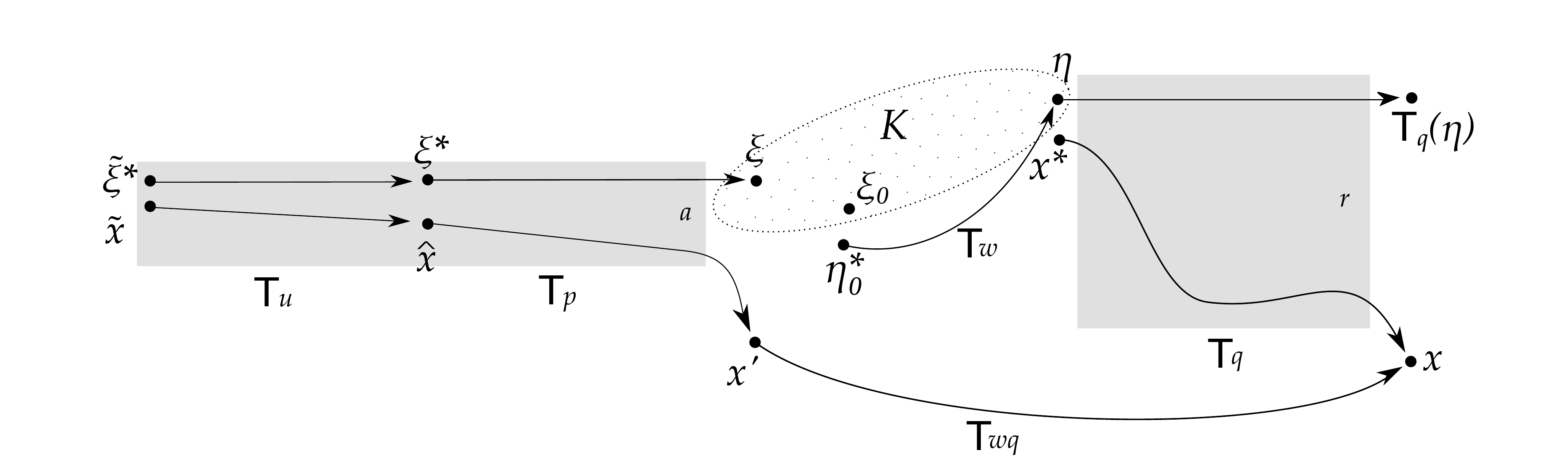} 
   \caption{Selection of preimages}
   \label{fig:comparable-construction}
\end{figure}
Decompose $v=upwq$ where $u,w,p,q$ are finite words and $|p|=|q|=n, |w|=m$.
Note that
$$
L_v(\mathbf 1)(x)=L_{wq}(L_{up}(\mathbf 1))(x)
\le \sup_{i\in\cW}\|L_i(\mathbf 1)\|_\infty^{n+m}\sup_{x'\in T^{-1}_{wq}(x)}\! L_{up}(\mathbf 1)(x').
$$
Fix $x'\in T^{-1}_{wq}(x)$. For any $\tilde x\in T^{-1}_{up}(x')$, let $\hat x=T_u(\tilde x)$. There exist by condition (i) of finite aperiodicity, $\xi\in K$ and $\xi^\ast\in T_p^{-1}(\xi)$ such that $d_p(\hat x, \xi^\ast)<a$. Let $\tilde \xi^\ast=(T_u)_{\tilde x}^{-1}(\xi^\ast)$, the inverse branch defined  in \eqref{inverse branch}. Then using \eqref{eq:bounded_distortion}
$$e^{\varphi_{up}(\tilde x)}=e^{\varphi_u(\tilde x)}e^{\varphi_p(\hat x)}\le e^{C_\varphi a^\alpha+\varphi_{u}(\tilde \xi^\ast)} e^{na^\alpha+\varphi_p(\xi^\ast)}=e^{C_\varphi a^\alpha+na^\alpha}e^{\varphi_{up}(\tilde \xi^\ast)}.$$
Because $d_{up}(\tilde x, \tilde \xi^\ast)<a$ and $T_{up}(\tilde\xi^\ast)=\xi$, one has $\tilde x=(T_{up})^{-1}_{\tilde \xi^\ast}(x')$ and $\tilde \xi^\ast=(T_{up})^{-1}_{\tilde x}(\xi)$. Therefore different $\tilde x$ is associated to different $\tilde \xi^\ast$, so that
$$L_{up}(\mathbf 1)(x')=\sum_{\tilde x\in T_{up}^{-1}(x')} e^{\varphi_{up}(\tilde x)}\ll \sum_{\tilde x\in T_{up}^{-1}(x') } e^{\varphi_{up}(\tilde \xi^{\ast})}\le \sum_{\xi\in K} L_{up}(\mathbf 1)(\xi).$$
Hence $$L_{v}(\mathbf 1)(x)\ll \sum_{\xi\in K}L_{up}(\mathbf 1)(\xi).$$
On the other hand, there exist by condition (ii) of finite aperiodicity a preimage $x^\ast\in T_q^{-1}(x)$ and  $\eta\in K$ such that $d(x^\ast,\eta)<a$ and $\eta\in B_{r}^q(x^\ast).$ As $d(x^\ast, \eta)<a$, we know that $L_{upw}(\mathbf 1)(x^\ast)\asymp L_{upw}(\mathbf 1)(\eta)$. Then
$$L_v(\mathbf 1)(x)\ge  e^{\varphi_q(x^\ast)}L_{upw}(\mathbf 1)(x^\ast)\gg e^{\varphi_q(\eta)-nr^\alpha}L_{upw}(\mathbf 1)(\eta)\gg L_{upw}(\mathbf 1)(\eta).$$
The last estimate holds because $q\in\cW^n$ and $\eta\in K$ both range over finite sets. Now for any $\xi\in K$ one can find $\eta^*\in T_w^{-1}(\eta)$ such that $d(\eta^*, \xi)<a$, then find  such a $\eta^*_0$ for $\xi_0$ that achieves $\max_{\xi_\in K}L_{up}(\mathbf 1)(\xi)$. Then $L_{up}(\mathbf 1)(\xi_0)\asymp L_{up}(\mathbf 1)(\eta^\ast_0)$ and 
\begin{align*}
L_{upw}(\mathbf 1)(\eta)&=\sum_{\eta^*\in T_w^{-1}(\eta)}e^{\varphi_w(\eta^\ast)} L_{up}(\mathbf 1)(\eta^*)\ge e^{\varphi_w(\eta_0^\ast)}L_{up}(\mathbf 1)(\eta_0^\ast)\\
&\gg e^{\varphi_w(\eta_0^\ast)}L_{up}(\mathbf 1)(\xi_0)\gg L_{up}(\mathbf 1)(\xi_0).
\end{align*}
The last estimate holds because $\varphi_w$ is continuous, $\eta_0^\ast\in \overline{B_a(\xi_0)}$, $\xi_0\in K$ and $w\in \cW^m$ range over finite sets. Therefore
$$L_v(\mathbf 1)(x)\gg \max_{\xi\in K} L_{up}(\mathbf 1)(\xi).$$

All the constants absorbed into $\ll$ or $\gg$ are determined by $\mathcal S, \varphi, K, m, n$ (essentially by $\mathcal S$ and $\varphi$), in particular independent of $v, x, y$. 
 It follows from the above estimates that $L_v(\mathbf{1})(x) \asymp L_v(\mathbf{1})(y)$ for any $x,y\in X$. 

Lastly when  $|v|\le 2n+m$, take any finite word $|v'|>2n+m$, then for any $x\in X$ 
\begin{align*}
L_{v'v}(\mathbf 1)(x)=L_v(L_{v'}(\mathbf 1))(x)=\sum_{\tilde x\in T_v^{-1}(x)} e^{\varphi(\tilde x)}L_{v'}(\mathbf 1)(\tilde x)
&\asymp \sum_{\tilde x\in T_v^{-1}(x)} e^{\varphi(\tilde x)}L_{v'}(\mathbf 1)(x)\\
&= L_v(\mathbf 1)(x)L_{v'}(\mathbf 1)(x)
\end{align*}
by the already-proven case. So $L_v(\mathbf 1)(x)\asymp L_{v'v}(\mathbf 1)(x)/L_{v'}(\mathbf 1)(x)$, hence
for any $x, y\in X$, $L_v(\mathbf 1)(x)\asymp L_v(\mathbf 1)(y)$.
\end{proof}

\section{Contraction in the Vaserstein distance} 
Let $\mathcal{M}_1(X)$ refer to the space of Borel probability measures on $X$. 
Recall that the Vaserstein distance $W$ of $\mu, \nu \in \mathcal{M}_1(X)$ 
defined by
\[ W(\mu, \nu) := \inf\left\{ \int d(x,y) dP: P \in \Pi(\mu, \nu) \right\} \]
is a compatible metric with weak convergence, where $\Pi(\mu, \nu)$ refers to the couplings of $\mu$ and $\nu$, that is the set of probability measures on $X \times X$ with marginal distributions $\mu$ and $\nu$. Moreover, by Kantorovich's duality,  
\[ W(\mu, \nu) = \sup \left\{ \Big| \int f  d(\mu-\nu)\Big|: D_1(f)\leq 1 \right\}.\]
Let ${\mathbb P_u^v}^\ast$ denote the dual operator of $\mathbb P_u^v$ on $\mathcal M_1(X)$. In order to obtain a contraction of $W({\mathbb{P}_{u}^{v}}^\ast(\cdot), {\mathbb{P}_{u}^{v}}^\ast(\cdot))$, 
the estimates of Lemma \ref{lem:doeblin-fortet} indicate that for $a$-close measures, one should consider $(d(x,y))^\alpha$ instead of $d(x,y)$. However, for distant measures, the method of proof below based on an idea in  \cite{HairerMattingly:2008} (see also  \cite{Stadlbauer:2015,KloecknerLopesStadlbauer:2015,BessaStadlbauer:2016,StadlbauerZhang:2017a}) requires a truncated distance. We consider 
\begin{equation} \label{eq:d star}  {d}^\ast(x,y) := \min\left\{1, \Delta \,{d(x,y)^\alpha} \right\}, \quad \Delta:= \max\{4C_\varphi,a^{-\alpha}\}\end{equation}
Observe that, by construction $d(x,y)< a$ whenever ${d}^\ast(x,y)<1$. In order to see that $d^\ast$ is a metric, observe that the triangle inequality follows from $x^\alpha + y^\alpha \geq (x+y)^\alpha$ for $x,y \geq 0$ and $0< \alpha \leq 1$, which is an inequality that easily can be deduced from the concavity of $x \mapsto x^\alpha$. The remaining assertion that $d^\ast(x,y) = 0$ if and only if $x=y$ is trivial. 

In order to define the  ${d}^\ast$-Lipschitz functions and their Lipschitz coefficients, set 
$$\overline{D}(f) := \max\left\{\sup_{x,y \in X}  |f(x)-f(y)|, D_\alpha^{\hbox{\tiny loc}}(f)/\Delta \right\},$$ where
\[D_\alpha^{\hbox{\tiny loc}}(f):=  \sup\left\{ \frac{|f(x)-f(y)|}{d(x,y)^{\alpha}} : x,y \in X, 0< d(x,y) < \Delta^{-\frac{1}{\alpha}}\right\}.  \]
As it can be easily seen, $\left\{f:  \overline{D}(f)< \infty \right\}$  is the space of ${d}^\ast$-Lipschitz functions  and, in particular, by Kantorovich's duality, the Vaserstein metric $\overline{W}$ with respect to $d^\ast$ is characterised  through local H\"older continuous
functions by
\[   
\overline{W}(\mu, \nu) = \sup \left\{ \Big| \int f  d(\mu- \nu) \Big|: 
\overline{D}(f) \leq 1  \right\}.
\] 
Note that $\overline D(f)\le 2\|f\|_\infty+ \Delta^{-1} D_\alpha(f)$, so functions in $\cH_\alpha$ have finite $d^*$-Lipschitz norms, and since $D_\alpha(f)\le \Delta \overline D(f)$, the norms $\|\cdot\|$ and $\|\cdot\|_\infty+\overline D(\cdot)$ are equivalent.
\begin{theorem}
\label{theorem:contraction} 
Suppose that  $\mathcal{S}$ is jointly topologically mixing and finitely aperiodic Ruelle-expanding semigroup, and that every potential $\varphi_i$ is $\alpha$-H\"older and summable.
Then there exist $k_0 \in \bbN$ and $s \in (0,1)$ such that for all finite words $u,v$ with $|v|\ge k_0$ and $\nu_1 , \nu_2\in \mathcal{M}_1(X)$ and $f$ with  $\overline{D}(f) < \infty$,
\begin{align*}
\overline{W}({\mathbb{P}_{u}^{v}}^\ast(\nu_1), {\mathbb{P}_{u}^{v}}^\ast (\nu_2)) &\leq  s^{n} \overline{W}( \nu_1 , \nu_2),\\
\quad \overline{D}(\mathbb{P}_{u}^{v}(f))  &\leq s^{n} \overline{D}(f).
\end{align*}
\end{theorem}
\begin{remark}\label{rmk:aperiodicity} Under the additional hypothesis that $X$ is compact, the  
condition of finite aperiodicity is automatically satisfied.
\end{remark}

\begin{proof} As in \cite{HairerMattingly:2008}, we first prove the assertions for 
 Dirac measures and then extend the partial result by optimal transport to arbitrary probability measures.
\subsubsection*{(1) Local contraction} Assume that ${d}^\ast(x,y)<1$ and that $f$ is $d^\ast$-Lipschitz continuous.  
Since $d(x,y)<a$ as soon as ${d}^\ast(x,y)<1$, Lemma \ref{lem:doeblin-fortet} gives that
\[ \mathbb{P}_{u}^{v}(f)(x) -   \mathbb{P}_{u}^{v}(f)(y) \leq \left(2C_\varphi \|f\|_\infty + \lambda^{|v|} D_\alpha^{\hbox{\tiny loc}}(f)\right)(d(x,y))^\alpha . \]
Furthermore, as $\mathbb{P}_{u}^{v}(\mathbf{1})=\mathbf{1}$, one may suppose without loss of generality that $\inf f=0$, and therefore, $\|f\|_\infty \leq \overline{D}(f)$. Dividing by $\Delta$ and choosing $k_0$ such that  $\lambda^{k_0} \leq 1/4$, it follows that for $v$ with $|v| \geq k_0$ 
\[  \mathbb{P}_{u}^{v}(f)(x) -   \mathbb{P}_{u}^{v}(f)(y) \leq \left( \frac{\|f\|_\infty}{2} +  \frac{D_\alpha^{\hbox{\tiny loc}}}{4 \Delta} \right) d^\ast(x,y) \leq  \frac{3 \overline{D}(f)}{4}d^\ast(x,y). \]    
Hence, by Kantorovich's duality, 
\[\overline{W}({\mathbb{P}_{u}^{v}}^\ast(\delta_x), {\mathbb{P}_{u}^{v}}^\ast(\delta_y)) \leq \frac34 d^\ast(x,y)  = \frac34 \overline{W}(\delta_x,\delta_y).\]

\subsubsection*{(2) Global contraction}
If ${d}^\ast(x,y) =1$, an upper bound for $\overline{W}$ can be obtained by construction of a coupling based on finite aperiodicity. 
In order to do so, fix an open set $U$ of diameter smaller than $a/2$. Suppose $\mathcal S$ is $n_1$-finitely aperiodic and $K, r$ are given by finite aperiodicity. As $\mathcal S$ is jointly topologically mixing, one can find $n_2$ such that $T_{w}(U) \cap B_{a}(\xi) \neq \emptyset$ for all $w\in \cW^{n_2}$ and  $\xi \in K$ and that $\lambda^{n_2} <1/8$. Choose $n_3$ large such that 
$C_{n_3}:=\Delta (a\lambda^{n_3})^\alpha<1/2.$ Let $k_0= n_1+n_2+n_3$.

Let $n\ge k_0$. For $v\in \cW^n$, write $v=v_3v_2v_1$ where $|v_1|=n_1, |v_2|=n_2$ and $|v_3|\ge n_3$. 
For any $x\in X$, we will select a preimage $x^\#$ in $T_{v_2v_1}^{-1}(x)$ as below, illustrated in Figure \ref{fig:coupling}.
\begin{figure}[h] 
   \centering
   \includegraphics[width=0.7\textwidth]{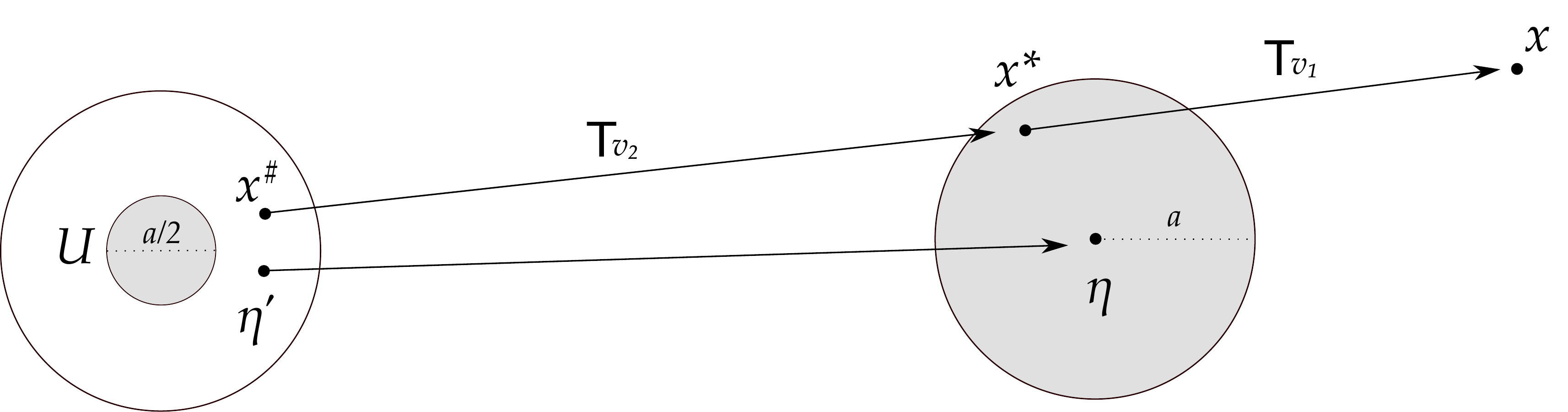} 
   \caption{The map $x\mapsto x^\#$}
   \label{fig:coupling}
\end{figure}

Let  $\eta\in K$ and ${x}^\ast \in X$ be given by condition (ii) of finite aperiodicity so that $T_{v_1}({x}^\ast)=x, d(x^\ast, \eta)<a$ and $x^\ast\in B_{r}^{v_1}(\eta)$. 
Now the choice of $n_2$ and Ruelle expanding property allow us to find a preimage $\eta'\in T_{v_2}^{-1}(\eta)$ such that $\eta'\in B_{a/8}(U)$.
Use Ruelle expanding property again to find a preimage $x^\#\in T_{v_2}^{-1}(x^\ast)$ such that $x^\# \in B_{a/8}(\eta')\subset B_{a/4}(U)$. One has $|\varphi_{v_2}(x^\#)-\varphi_{v_2}(\eta')|\le C_{\varphi} a^\alpha$ by \eqref{eq:bounded_distortion}. So that $$ |\varphi_{v_2v_1}(x^\#) - \varphi_{v_2v_1}(\eta')| \leq C_\varphi a^\alpha+n_1 r^\alpha \max_{i\in \cW} D_\alpha(\varphi_i),$$ 
and hence $$e^{ \varphi_{v_2v_1}(x^\#)} \asymp e^{\varphi_{v_2v_1}(\eta')}=e^{\varphi_{v_2}(\eta')}e^{\varphi_{v_1}(\eta)}.$$
Since $\eta'$ lies in a fixed bounded region $B_{a/8}(U)$ and $\varphi$ is continuous and  $\eta\in K,  v_1\in \cW^{n_1}, v_2\in \cW^{n_2}$ range over finite sets, one concludes that for all $x\in X, v_1\in\cW^{n_1}, v_2\in\cW^{n_2}$
\begin{equation}\label{eq:contraction-preimg}
e^{ \varphi_{v_2v_1}(x^\#)}\asymp 1.
\end{equation}

Now for any pair $(x,y)\in X^2$, find as before $x^\#, y^\#\in B_{a/4}(U)$. Then $d(x^\#, y^\#)<a$. As stated in \eqref{bijection of inverse}, there is a bijection $\tilde x\mapsto \tilde y$ from $T_{v_3}^{-1}(x^\#)$ to $T_{v_3}^{-1}(y^\#)$. Pair $(\tilde x, \tilde y)$ together by this bijection and set a subprobability measure on $X^2$
$$Q_{(x,y)} := \min \left\{\sum_{(\tilde x, \tilde y)}    \frac{e^{\varphi_{v}(\tilde x)}L_u(\mathbf{1})(\tilde x)}{L_{uv}(\mathbf{1})(x)}  \; \delta_{(\tilde{x},\tilde{y})},\  \sum_{(\tilde x, \tilde y)}   \frac{e^{\varphi_{v}(\tilde y)}L_u(\mathbf{1})(\tilde y)}{L_{uv}(\mathbf{1})(y)} \; \delta_{(\tilde{x},\tilde{y})}\right\}.$$
Note that $Q_{(x,y)}(X^2)=Q_{(x,y)}(\{(z_1, z_2):d(z_1, z_2)<a\lambda^{|v_3|}\})$. For any $A\subset X$ $$Q_{(x,y)}(A\times X)\le \sum_{T_v(z)=x} \frac{e^{\varphi_v(z)}\mathbf 1_A\cdot L_u(\mathbf 1)(z)}{L_{uv}(\mathbf 1)(x)}=\frac{L_v(\mathbf 1_A\cdot L_u(\mathbf 1))}{L_{uv}(\mathbf 1)}(x)={\mathbb P_{u}^{v}}^\ast(\delta_x)(A)$$
and similarly $Q_{(x,y)}(X\times A)\le {\mathbb{P}_{u}^{v}}^\ast(\delta_y)(A).$ Hence, there exists a further subprobability measure $R$ such that $P:= Q_{(x,y)}+R \in \Pi({\mathbb{P}_{u}^{v}}^\ast(\delta_x), {\mathbb{P}_{u}^{v}}^\ast(\delta_y))$ (see, e.g., \cite{HairerMattingly:2008}). Therefore, due to the choice of $n_3$,
\begin{align*}
 &\quad \overline{W}({\mathbb{P}_{u}^{v}}^\ast(\delta_x), {\mathbb{P}_{u}^{v}}^\ast(\delta_y)) 
 \leq  \int d^\ast(z_1, z_2) dP \\
 &\leq  \Delta (a \lambda^{|v_3|})^\alpha P(\{ d(z_1, z_2)< a \lambda^{|v_3|} \}) + P(\{ d(z_1, z_2)\geq  a \lambda^{|v_3|} \})\\
 &\leq 1- C_{n_3} P(\{ d(z_1, z_2)< a \lambda^{|v_3|} \}) \leq 1 - C_{n_3} Q_{(x,y)}(X^2).
 \end{align*}
To get a lower bound for $Q_{(x,y)}(X^2)$, use \eqref{eq:contraction-preimg} to see
\begin{align*}
Q_{(x,y)} (X^2) &\asymp \min \left\{\sum_{T_{v_3}(\tilde x)=x^\#}    \frac{e^{\varphi_{v_3}(\tilde x)}L_u(\mathbf{1})(\tilde x)}{L_{uv}(\mathbf{1})(x)} , \sum_{T_{v_3}(\tilde y)=y^\#}   \frac{e^{\varphi_{v_3}(\tilde y)}L_u(\mathbf{1})(\tilde y)}{L_{uv}(\mathbf{1})(y)} \right\}\\
&=\min\left\{\frac{L_{uv_3}(\mathbf 1)(x^\#)}{L_{uv}(\mathbf 1)(x)},\  \frac{L_{uv_3}(\mathbf 1)(y^\#)}{L_{uv}(\mathbf 1)(y)} \right\}.
\end{align*}
Applying Lemma \ref{comparability} we get that for any $\xi_0\in K$
$$Q_{(x,y)}(X^2)\asymp \frac{1}{L_{v_2v_1}(\mathbf{1})(\xi_0)} \geq \min \{ (L_w(\mathbf{1})(\xi))^{-1} : \xi \in K, w\in \cW^{n_1+n_2}\}>0.$$
Hence, there is a lower bound $N \leq Q_{(x, y)}(X^2)$, independent of $x,y \in X$ and $v\in \cW^n$.  Therefore, increasing $n_3$ so that $C_{n_3}N<1$ if needed, 
$$\overline{W}({\mathbb{P}_{u}^{v}}^\ast(\delta_x), {\mathbb{P}_{u}^{v}}^\ast(\delta_y))  \leq    1 -   C_{n_3} N  
= (1 -  C_{n_3} N ) d^\ast(x,y) = (1 -   C_{n_3}N ) \overline{W}( \delta_x,\delta_y).
$$ 

Combining part (1) with part (2) of the proof and letting $t:= \max\{3/4, 1 - C_{n_3}N\}<1$, we obtain that  there exists $k_0$ such that for all finite words $u, v$ with $|v|\ge k_0$ and $x,y\in X$
\[\overline{W}({\mathbb{P}_{u}^{v}}^\ast(\delta_x), {\mathbb{P}_{u}^{v}}^\ast(\delta_y)) \leq t \overline{W}(\delta_x,\delta_y).\]
Using Kantorovich's duality, for $f$ with $\overline{D}(f) \leq 1$, it follows that  
\[ 
 |\mathbb{P}_{u}^{v}(f)(x) -  \mathbb{P}_{u}^{v}(f)(y)| = \left| \int f d{\mathbb{P}_{u}^{v}}^\ast(\delta_x) -  \int f d{\mathbb{P}_{u}^{v}}^\ast(\delta_y)  \right| \leq t.
\] 

\subsubsection*{(3) Contraction for arbitrary probability measures} The extension to arbitrary probability measures is a standard application of optimal transport and omitted as the proof is a straightforward adaption of \cite{HairerMattingly:2008}, \cite {Stadlbauer:2015} or \cite{KloecknerLopesStadlbauer:2015}.  We obtain that for any finite words $u, v$ with $|v|\ge k_0$ and any probability measures $\nu_1, \nu_2$
\[\overline{W}({\mathbb{P}_{u}^{v}}^\ast(\nu_1), {\mathbb{P}_{u}^{v}}^\ast(\nu_2)) \leq t \overline{W}(\nu_1,\nu_2).\]
\subsubsection*{(4) Iteration}
By the iteration rules given in \eqref{eq:iteration_1}, the theorem follows for $s = t^{1/2k_0}$.
\end{proof}

\section{Conformal measures, quenched exponential decay and continuity} 
From now on we always assume that $\mathcal{S}$ is jointly topologically mixing and finitely aperiodic and every potential $\varphi_i$ is $\alpha$-H\"older and summable, so that Theorem \ref{theorem:contraction} holds.
It has immediate consequences for the existence and regularity of two types of compact sets of probability measures, which are canonical generalisations of conformal measures and equilibrium states to the context of semigroups. 

\subsection{One-sided dynamics}

Denote by $\Sigma=\{i_1i_2\ldots: i_1, i_2,\ldots\in \cW\}$ the set of infinite words and by  $\theta(i_1i_2\ldots)=i_2i_3\ldots$ the shift map. For an infinite word $\omega=i_1 i_2\ldots   \in \Sigma$ and $k \in \mathbb{N}$, let
\[ [\omega]_k:=i_1 \ldots i_k \in\cW^k.\]
The first family of measures is constructed as follows, which generalises the notion of conformal measures.  
\begin{proposition} \label{prop:conformal}
For  any finite word $u$, infinite word $\omega$ and measure $\nu\in \mathcal{M}_1(X)$, the limit 
\[\mu_{u,\omega}:=  \lim_{l \to \infty}  {\mathbb{P}_{u}^{[\omega]_l}}^\ast(\nu) \]
exists and is independent of $\nu$. Furthermore, with $k_0$ and $s$ given by Theorem \ref{theorem:contraction},
 the following statements hold.
\begin{enumerate}[label={\rm(\roman*)}, itemsep=0pt, topsep=0pt]
\item For $k\geq k_0$ and any $\omega, \tilde{\omega}  \in \Sigma$ with $[\omega]_k = [\tilde{\omega}]_k$,
$\overline{W}(\mu_{u,\omega}, \mu_{u,\tilde{\omega}}) \leq s^k$.
\item   For  $k\geq k_0$ and $f \in \mathcal{H}_\alpha$,
\[ \left\| \mathbb{P}_{u}^{[\omega]_k}(f) - \int f d\mu_{u,\omega} \right\| \leq 2 s^k  \overline{D}(f) .\]
\item Let $\mu_{\omega}:= \mu_{\emptyset, \omega}$, then
\[ \mu_{u\omega}={\mathbb P^u}^*(\mu_{u,\omega}),\quad \mu_{u,\omega} = \mu_{u\omega}\circ T_u^{-1}. \]
If $v$ is a finite word, $$\mu_{u, v\omega}={\mathbb P_{u}^{v}}^\ast (\mu_{uv, \omega}).$$
\item Let $\lambda_{u, \omega} := \int L_u(\mathbf{1}) d\mu_\omega $, then
$$
L_u^\ast(\mu_\omega) = \lambda_{u,\omega} \mu_{u\omega},$$
and if $v$ is a finite word,
$$\lambda_{uv,\omega} = \lambda_{u,v\omega} \lambda_{v,\omega}.$$
\item The measures $\mu_{u, \omega}$ and $\mu_{\omega}$ are absolutely continuous to each other and
$$h_{u,\omega}:=\frac{d\mu_{u,\omega}}{d\mu_{\omega}}=\lambda_{u,\omega}^{-1} L_u(\mathbf{1}).$$
\end{enumerate}
\end{proposition}

\begin{proof}
For probability measures $\nu,\tilde{\nu}$ on $X$ and $l >k\ge k_0$,  Theorem \ref{theorem:contraction} implies  
\[ \overline{W}\left({\mathbb{P}_{u}^{[\omega]_k}}^\ast(\nu), {\mathbb{P}_{u}^{[\omega]_l}}^\ast(\tilde{\nu})\right) =  \overline{W}\left({\mathbb{P}_{u}^{[\omega]_k}}^\ast(\nu), {\mathbb{P}_{u}^{[\omega]_k}}^\ast\circ {\mathbb{P}_{u[\omega]_k}^{[\theta^k \omega]_{l-k}}}^\ast (\tilde{\nu})\right)\leq s^k. \]
Hence, $\{{\mathbb{P}_{u}^{[\omega]_k}}^\ast(\nu)\}_{k\geq k_0}$ is a Cauchy sequence and $\mu_{u,\omega} := \lim_k {\mathbb{P}_{u}^{[\omega]_k}}^\ast(\nu)$ exists and is independent of $\nu$. This, in particular, implies the estimate in (i). In order to show (ii), it suffices to consider $\nu=\delta_x$. If $k \geq k_0$,  we have that
$$\left|\mathbb{P}_{u}^{[\omega]_k}(f)(x)- \int f d\mu_{u,\omega}\right| \leq \overline{D}(f) s^k.$$ The estimate in  (ii) then follows from this combined with Theorem \ref{theorem:contraction}. 

The second part of (iii) follows from
$$\int \mathbb P_u^v(f)d\mu_{uv,\omega} =\lim_{k \to \infty} \mathbb P_{uv}^{[\omega]_k}\circ \mathbb P_u^v(f)(x)=\lim_{k\to\infty}\mathbb P_{u}^{v[\omega]_k}(f)(x)=\int f d\mu_{u, v\omega}.$$
The first part of (iii) follows from this and
\begin{align*}
\int f d\mu_{u,\omega} &= \lim_{k \to \infty} \frac{L_{[\omega]_k} (f L_u(\mathbf{1}))(x)}{ L_{u[\omega]_k}(\mathbf{1})(x)}  =  \lim_{k \to \infty} \frac{L_{u[\omega]_k}(f\circ T_u)(x)}{L_{u[\omega]_k}(\mathbf{1})(x)}\\
& = \int f\circ T_u d\mu_{u\omega} = \int fd\mu_{u\omega}\circ T_u^{-1}.
\end{align*}
(iv) holds because
\begin{align*}
\int L_u(f)d\mu_\omega & =   \lim_{k \to \infty} \frac{L_{[\omega]_k} (L_u(f))(x)}{ L_{[\omega]_k}(\mathbf{1})(x)} 
=  \lim_{k \to \infty} \frac{L_{u[\omega]_k} (f)(x)}{ L_{u[\omega]_k}(\mathbf{1})(x)}\cdot \frac{ L_{u[\omega]_k}(\mathbf{1})(x)}{ L_{[\omega]_k}(\mathbf{1})(x)}\\
& = \int f d\mu_{u\omega} \int L_u(\mathbf{1}) d\mu_\omega
\end{align*}
and
\begin{align*}
\lambda_{uv,\omega}\mu_{uv\omega}=L_{uv}^\ast(\mu_\omega)=L_u^\ast L_v^\ast(\mu_\omega)=L_u^\ast(\lambda_{v,\omega}\mu_{v\omega})=\lambda_{v,\omega}\lambda_{u,v\omega}\mu_{uv\omega}.
\end{align*}
(v) follows from
$$\int f  d\mu_{u,\omega}  = \lim_{k \to \infty} \frac{L_{[\omega]_k}(\mathbf{1})(x)}{L_{u[\omega]_k}(\mathbf{1})(x)} \cdot  \frac{L_{[\omega]_k}(f L_u(\mathbf{1}))(x)}{L_{[\omega]_k}(\mathbf{1})(x)} = \frac{1}{\lambda_{u,\omega}}  \int f L_u(\mathbf{1}) d\mu_\omega.$$
\end{proof}

\begin{remark} Recall that a probability measure $\nu$ is \emph{$(T_w,\varphi_w)$-conformal}, where $w$ is a finite word, if there exists $c > 0$ such that $L_w^\ast(\nu)=c\nu$. Consider $\overline{w}:= ww\ldots \in \Sigma$ and $\mu_{\overline{w}}=\mu_{\emptyset, \overline w}$ given by Proposition  \ref{prop:conformal}. By (iv) of the same proposition,  $L_w^\ast(\mu_{\overline{w}})=\lambda_{w,\overline{w}} \mu_{\overline{w}}$, hence $\mu_{\overline{w}}$ is conformal. Moreover, (i) and $\mu_{u\overline{w}}\circ T_u^{-1} = \mu_{u,\overline{w}}$ imply
\[ \left\{\mu_{u,\omega} :\omega \in \Sigma\right\} =  \overline{\left\{ \mu_{u\overline{w}} \circ T_u^{-1} : w\in\cup_{k\geq 1} \cW^k\right\}}. \]
As $\Sigma$ is compact and $\omega \mapsto \mu_{u,\omega}$ is Lipschitz continuous by (i) of Proposition  \ref{prop:conformal}, $\{\mu_{u,\omega} : \omega\in\Sigma\}$ is compact. 
It is also worth mentioning that item (i) ensures that, for any finite word $u$, the family $\Sigma\ni \omega \mapsto \mu_{u,\omega}$
is H\"older continuous. Finally, the fact that any two asymptotic limits are equivalent (recall item (v)) will be useful to provide an 
application to characterize the boundary of a semigroup action in Section~\ref{sec:app}. 
\end{remark}

\subsection{Two-sided compositions} 

The second family of probabilities one will consider generalises the notions of invariant measures and equilibrium states. 
To attain that goal, despite the fact that the underlying dynamics is not invertible, we need to consider forward iterations of 
maps determined by two-sided sequences.
 Let $\Sigma^-$ refer to the set of left-infinite words, that is 
  $ \Sigma^- = \left\{ \ldots i_2 i_1 : i_1, i_2,\ldots \in \cW \right\},$
and for $k \in \mathbb{N}$ and $\sigma=\ldots i_2 i_1  \in \Sigma^-$ define
\[ _k[\sigma] := i_k \ldots  i_2 i_1\in\cW^k.\]
\begin{proposition} \label{prop:equilibrium}
For  any $\sigma \in \Sigma^-$, $\omega \in \Sigma$ and $\nu\in \mathcal{M}_1(X)$, the limit
\[
 \mu_{\sigma,\omega}:= \lim_{k,l \to \infty}  {\mathbb{P}_{_k[\sigma]}^{[\omega]_l}}^\ast(\nu) 
\]
exists and is independent of $\nu$. Furthermore, with $k_0$ and $s$ given by Theorem \ref{theorem:contraction},
 the following statements hold.
\begin{enumerate}[label={\rm(\roman*)}, itemsep=0pt, topsep=0pt]
\item 
For $k,l$ with $k\wedge l\ge k_0$ and $\sigma,\tilde{\sigma} \in \Sigma^-, \omega,\tilde{\omega} \in \Sigma$ with $_k[\sigma]={_k[\tilde{\sigma}]}, [\omega]_l=[\tilde{\omega}]_l$,
$\overline{W}(\mu_{\sigma, \omega}, \mu_{\tilde{\sigma},\tilde{\omega}}) \leq s^{k\wedge l}$.
\item   For  $k,l$ with $k\wedge l\ge k_0$ and $f\in \mathcal{H}_\alpha$,
\[ \left\| \mathbb{P}_{_k[\sigma]}^{[\omega]_l}(f) - \int f d\mu_{\sigma, \omega} \right\| \leq 2 s^{k\wedge l}  \overline{D}(f).\]
\item For a finite word $u$, $\mu_{\sigma u,\omega} = \mu_{\sigma, u\omega}\circ T_u^{-1}$.
\item \label{propitem:continnuity of h} The measures $\mu_{\sigma, \omega}$ and $\mu_\omega$ are absolutely continuous to each other and 
$h_{\sigma, \omega}:= d\mu_{\sigma, \omega}/d\mu_\omega$ satisfies
\[  \left\| h_{_k[\sigma],\omega} - h_{\sigma, \omega} \right\| \ll s^k,\]
where $\mu_\omega$ and $h_{_k[\sigma], \omega}$ are as given in the previous proposition. 
\end{enumerate}
\end{proposition}
\begin{proof} 
As a consequence of Proposition \ref{prop:conformal} (ii), Lemma \ref{lem:doeblin-fortet} and Lemma \ref{comparability}, for any finite word $u$, infinite word $\omega \in \Sigma$ and $l \geq k_0$ we have that 
\begin{equation} \label{eq:estimate-lambda} \left\| L_{u[\omega]_l}(\mathbf{1})/  L_{[\omega]_l}(\mathbf{1}) - \lambda_{u,\omega} \right\| \leq s^l \overline{D}(L_u(\mathbf{1})) \leq C s^l \lambda_{u,\omega}, \end{equation}
for some $C>0$. Hence, for finite words $v \in \mathcal{W}^k$, $w \in \mathcal{W}^l$, $k \geq k_0$ and $f$ H\"older continuous, 
\begin{align*}
 & \left| \mathbb{P}_{v}^{w}(f) -   \mathbb{P}_{uv}^{w}(f)  \right| \\
\leq& \left| 
\frac{L_{w}(f L_{v}(\mathbf{1}))}{L_{vw}(\mathbf{1}) }
- 
\frac{L_{w}(f L_{uv}(\mathbf{1}))}{\lambda_{u,\overline{vw}} L_{vw}(\mathbf{1}) }
\right| +
\left| \frac{L_{w}(f L_{uv}(\mathbf{1}))}{\lambda_{u,\overline{vw}} L_{vw}(\mathbf{1}) }
- \frac{L_{w}(f L_{uv}(\mathbf{1}))}{L_{uv w}(\mathbf{1}) } \right| 
\\
  \leq &
 \frac{L_{w}\left(|f| L_{v}(\mathbf{1})  \left| 1- \frac{L_{uv}(\mathbf{1})}{ \lambda_{u,\overline{vw}}L_{v}(\mathbf{1})  }\right|   \right)}{L_{vw}(\mathbf{1})   }
 +
 \frac{L_{w}(|f| L_{uv}(\mathbf{1}))}{L_{uv w}(\mathbf{1}) } 
 \left| 
 \frac{L_{uv w}(\mathbf{1})}{ \lambda_{u,\overline{vw}}L_{vw}(\mathbf{1}) } -1 
 \right| \\
\leq  & C \left( \mathbb{P}_{v}^{w}(|f|) s^k + \mathbb{P}_{uv}^{w}(|f|) s^{k+l} \right),
\end{align*}
where we used the notation $\overline{u}:=(u u \dots)$ to denote the periodic word formed by $u$ blocks.
Now assume that $\nu$ and $\tilde{\nu}$ are probability measures and $f$ is H\"older continuous with $\overline{D}(f) \leq 1$ and $\inf_{x \in X} f(x) =0$. In particular, $\|f\|_\infty \leq 1$.  By the above and Proposition \ref{prop:conformal}, for $\sigma,\tilde{\sigma} \in \Sigma^-$ and $\omega,\tilde{\omega} \in \Sigma$ such that 
$_k[\sigma]={_k[\tilde{\sigma}]}, [\omega]_l=[\tilde{\omega}]_l$ and $k\wedge l\geq k_0$,
\begin{align*}
 & \left| \int \mathbb{P}_{_k[\sigma]}^{[\omega]_{l}}(f) d\nu  - \int \mathbb{P}_{_k[\tilde{\sigma}] }^{[\tilde{\omega}]_l}(f)  d\tilde{\nu} \right| \\
 \leq & 
 \int   \left| \mathbb{P}_{_k[\sigma]}^{[\omega]_{l}}(f) -  \mathbb{P}_{_k[\tilde{\sigma}]}^{ [\omega]_{l}}(f)  \right| d\nu 
  +
  \left| \int \mathbb{P}_{_k[\tilde{\sigma}]}^{[\omega]_{l}}(f) d\nu -  \int \mathbb{P}_{_k[\tilde{\sigma}]}^{[\omega]_{l}}(f) d\tilde{\nu} \right| \\
 \leq & 
 C  \left( 
 2 \|\mathbb{P}_{_k[\sigma]}^{[\omega]_{l}}(f)\|_\infty s^k + \|\mathbb{P}_{_k[\sigma]}^{[\omega]_{l}}(f)\|_\infty s^{k+l}  + 
  \|\mathbb{P}_{_k[\tilde{\sigma}]}^{[\omega]_{l}}(f)\|_\infty s^{k+l} 
 \right) + 2 s^{l} \\
 \leq & 2C(s^k + s^{k+l}) + 2s^l \ll s^{k\wedge l}  .
\end{align*}
Hence, by Kantorovich's duality and completeness of the space of probability measures, $\lim_{k,l\to\infty} {\mathbb{P}_{_k[\sigma]}^{[\omega]_{l}}}^\ast (\nu)$ exists, is independent of $\nu$ and the estimate in part (i) holds. Part (ii) is an immediate consequence of part (i), and the proof of (iii) follows as in Proposition \ref{prop:conformal}. Proposition \ref{prop:conformal} (v) indicates that $h_{\sigma,\omega}$ is the limit of $h_{_k[\sigma],\omega}$ and by the first argument in Proposition 2.2 in \cite{BessaStadlbauer:2016}, it follows that $\|h_{_k[\sigma],\omega} - h_{_l[\sigma],\omega}\|_\infty \ll s^{k\wedge l}$. Then the argument in there can be easily adapted to obtain exponential convergence with respect to $\|\cdot \|_{d^\ast}$ in part (iv).
\end{proof}

\begin{remark}\label{remark-adic flow} The first part of the above proposition implies that the map $(\sigma, \omega) \mapsto \mu_{\sigma, \omega}$ is Lipschitz continuous with respect to the metric 
\[d((\sigma, \omega), (\tilde{\sigma},\tilde{\omega})) := \min\{ s^{k\wedge l}: {_k[\sigma]}={_k[\tilde{\sigma}]}, [\omega]_l=[\tilde{\omega}]_l\}.\]
In particular, the image of each compact subset of $\Sigma^-\times\Sigma$ is a compact subset of the space of probability measures. 

Moreover, by fixing an order on $\cW$, the associated adic flow $h_t$ on $\Sigma^-\times \Sigma$ is uniquely ergodic (see \cite{Fisher:2004a}) and, in particular, for any H\"older continuous $f:X \to \mathbb R$, the continuity of $(\sigma, \omega) \to \int f d\mu_{\sigma, \omega}$ implies that 
\[ \frac1T \int_0^T\int f(x) d\mu_{h_t(\sigma, \omega)}(x) dt \xrightarrow{T \to \infty} \iint  f(x) d\nu_{\sigma, \omega}(x) dm(\sigma, \omega) \] uniformly, where $m$ refers  to the Parry measure (or measure of maximal entropy). The analogue of this statement holds for $\omega \to \int f d\mu_{\omega, \omega}$ and Birkhoff sums with respect to the odometer on $ \Sigma$, or with respect to uniformly ergodic adic flows or adic transformations acting on compact subsets of $\Sigma^-\times \Sigma$ or $\Sigma$, respectively.
\end{remark}

The result provides the following link to invariant measures and equilibrium states. A finite word $w$ generates a periodic infinite word $\overline{w}:= (ww\ldots) \in \Sigma$ and a periodic left-infinite word $\underline w:=(\ldots ww)\in \Sigma^-$. Then, by Proposition  \ref{prop:equilibrium}, the measure $\mu_{\underline{w},\overline{w}}$ is $T_w$-invariant,  $d \mu_{\underline{w},\overline{w}} = h_{\underline{w},\overline{w}}d\mu_{\overline{w}}$ and
\[L_{w}(h_{\underline{w},\overline{w}}) = \lambda_{w,\overline{w}} h_{\underline{w},\overline{w}}.\] Here $\lambda_{w, \overline w}$ is given as in Proposition \ref{prop:conformal}.

The following result identifies $\mu_{\underline{w},\overline{w}}$ as the unique equilibrium state of $T_w$ with respect to the 
H\"older potential $\varphi_w$. Note that the statement avoids the notion of pressure as $X$ might be non-compact. However, if $X$ is compact, then $\log  \lambda_{w,\overline{w}}$ is equal to the pressure (\cite{Ruelle:1989}) and one obtains the usual notion of equilibrium state.  
In the proposition, $H_{\mu}(T_w)$ refers to Kolmogorov's entropy.
\begin{proposition} \label{prop:eq state}
\begin{align*} \log  \lambda_{w,\overline{w}} & 
{=}  H_{\mu_{\underline{w},\overline{w}}}(T_w) +  \int \varphi_w  d\mu_{\underline{w}, \overline{w}} \\
\nonumber & = \sup \left\{ H_{\nu}(T_w) +   \int \varphi_w  d\nu: \nu \in \mathcal{M}_1(X),  \nu = \nu \circ T_w^{-1}  \right\}.
\end{align*}
Furthermore, $\mu_{\underline{w},\overline{w}}$ is the unique measure which realises the supremum. 
\end{proposition}

\begin{proof} 
Let $J_{\mu_{\underline{w},\overline{w}}} := {d \mu_{\underline{w},\overline{w}}\circ T_w}/{d \mu_{\underline{w},\overline{w}}}$, and let
$$\tilde{\varphi}_w := \varphi_w + \log h_{\underline{w},\overline{w}} - \log h_{\underline{w},\overline{w}}\circ T_w - \log  \lambda_{w,\overline{w}}.$$
 By construction, $J_{\mu_{\underline{w},\overline{w}}}  = \exp(-\tilde{\varphi}_w)$ and, as $T_w$ is Ruelle expanding, Rokhlin's formula for entropy implies that
 \begin{align*}   H_{\mu_{\underline{w},\overline{w}}}(T_w)   & = \int \log J_{\mu_{\underline{w},\overline{w}}}   d\mu_{\underline{w},\overline{w}} \\ 
 & = \log  \lambda_{w,\overline{w}} - 
 \int (\varphi_w  + \log h_{\underline{w}, \overline{w}}  -  \log h_{\underline{w}, \overline{w}} \circ T_w) d\mu_{\underline{w}, \overline{w}}  \\
& = 
  \log  \lambda_{w,\overline{w}} - 
 \int \varphi_w  d\mu_{\underline{w}, \overline{w}}.
 \end{align*}
This proves the the first identity. Now suppose that $\nu$ is an invariant probability measure with $ H_{\nu}(T_w) +   \int \varphi_w d\nu  \geq  \log  \lambda_{w,\overline{w}}$. Then, by Rokhlin's formula, the invariance of $\nu$ and the definition of the transfer operator of $T_w$ with respect to $\nu$, 
\begin{align*} 
0 & \leq    H_{\nu}(T_w) +   \int \varphi_w  d\nu -  \log  \lambda_{w,\overline{w}}
\\
  & = \int (\log J_\nu +  \varphi_w + \log h_{\underline{w},\overline{w}} - \log h_{\underline{w},\overline{w}}\circ  T_w  -   \log  \lambda_{w,\overline{w}}) d\nu \\
  & = \int \log \frac{J_\nu}{J_{\mu_{\underline{w},\overline{w}}} }  d\nu 
 = \int \sum_{T_w(y)=x} \frac1{J_\nu(y)}  \log \frac{J_\nu(y)}{J_{\mu_{\underline{w},\overline{w}}}  (y)  } d\nu(x).
     \end{align*}
As $\nu$ is invariant, it follows that $\sum_{T_w(y)=x} 1/{J_\nu(y)} =1$ for all $x \in X$. Hence, by Jensens's inequality, 
\begin{align*} 
0 & \leq    H_{\nu}(T_w) +   \int \varphi_w  d\nu -  \log  \lambda_{w,\overline{w}} \stackrel{\ast}{\leq} 
\int  \log \sum_{T_w(y)=x} \frac1{J_\nu(y)}  \frac{J_\nu(y)}{J_{\mu_{\underline{w},\overline{w}}}  (y)  }d\nu(x) = 0. 
\end{align*}
Moreover, equality holds in $(\ast)$ if and only ${J_\nu(y)}/{J_{\mu_{\underline{w},\overline{w}}}(y)  }=1$ a.s.  
\end{proof}

\begin{remark}\label{remark:equi} 
By usual normalisation procedure, replacing the potential $\varphi_w$ with $\tilde{\varphi}_w $ one then obtains a new operator $\tilde{L}_w$ with $\tilde{L}_w(\mathbf{1}) = \mathbf{1}$, that is, $\tilde{L}_w$ is normalised and $\tilde{L}_w^\ast(\mu_{\underline w, \overline w})=\mu_{\underline w, \overline w}$. In particular, part (ii) of 
Proposition \ref{prop:conformal} applied to the semigroup generated by $T_w$ implies that 
$\tilde{L}_w$ has a spectral gap. However, the construction depends on the specific periodic word $\overline{w}$ and is in general not functorial, that is $\tilde{L}_{vw} \neq \tilde{L}_{w}\circ \tilde{L}_{v}$.
\end{remark}

\section{Annealed exponential decay} 

So far we have considered only quenched operators, which are determined by iterations in $\mathcal S$ tracked by certain finite words and their limiting behaviour. As stated in the introduction, another objective is to study annealed operators, which are averages of all the quenched operators tracked by finite words of given lengths. To be more precise, 
suppose that the one-sided full shift of finite alphabet $(\Sigma, \theta)$ is endowed with a non-singular probability measure $\rho$. For every $k\in\mathbb N$, define the \emph{averaged transfer operator}  
$$
\cA_k(f)(x)  := \int_\Sigma L_{[\omega]_k}(f)(x)\, d\rho(\omega) 
$$
for $f\in\cH_\alpha$. One can do so for more general shifts, but we keep $\Sigma$ to be a topological mixing subshift of finite type for simplicity. Naturally, one would need some properties of the shift space $(\Sigma, \theta, \rho)$ to study the operator $\mathcal A_k$. We summarise them below.

Since $\rho$ is non-singular, for a finite word $ u$, let $p_u:\Sigma\to\mathbb R_+$ be defined by $$p_u(\omega) := \frac{d\rho}{d\rho\circ\theta^{|u|}}(u\omega), \quad \omega\in\Sigma.$$
 With the usual distance given on the shift, denote by ${\cH}(\Sigma)$ the space of H\"older continuous functions on $\Sigma$ and by $\cC(\Sigma)$ the space of continuous functions on $\Sigma$. Recall that $\lambda_{u,\omega}=\int L_u(\mathbf 1)d\mu_\omega$ as in  Proposition \ref{prop:conformal}. Note that $\log  \lambda_{i,\cdot}\in\cH(\Sigma)$ by Proposition \ref{prop:conformal}. Suppose that $\log p_i\in\cH(\Sigma)$ as well.
Define a linear operator $\iota$ acting on $\cC(\Sigma)$ by
\[ \iota(g)(\omega) := \sum_{i\in \cW}  \lambda_{i,\omega}  p_i(\omega)  g(i\omega),\quad g\in \cC(\Sigma). \]
As  $u \mapsto p_u$ and $u \mapsto  \lambda_{u,\omega}$  are multiplicative cocycles with respect to $\theta$, it can be shown that for every $k\in\mathbb N$
$$ \iota^k(g)(\omega) =\sum_{u\in\cW^k}\lambda_{u,\omega}p_u(\omega) g(u\omega).$$
In view of the duality with $\theta$,  we  have that for any $g_1, g_2\in\cC(\Sigma)$
\begin{equation}\label{eq:iotadual}
 \int \iota^k(g_1)\cdot g_2 d\rho= \int \lambda_{[\omega]_k, \theta^k\omega}\cdot g_1\cdot g_2\circ \theta^k d\rho.
 \end{equation}

Since $\log  \lambda_{i,\omega}$ and $\log p_i$ are both H\"older continuous, Ruelle's Perron-Frobenius theorem implies that there are $\beta>0, m\in\mathcal M_1(\Sigma)$ and $g_o\in \cC(\Sigma), g_o>0$ such that 
\begin{equation}\label{eq:measurem}
\iota^*m=\beta m,\quad  \iota(g_o)=\beta g_o,\quad  m(g_o)=1.
\end{equation}

Furthermore there exists $t \in (0,1)$ such that for any $g\in \cH(\Sigma)$ and $k\in\mathbb N$
\begin{equation}\label{eq:iotadecay}
  \left\| \beta^{-k}\iota^k(g) -  g_o\int g \, dm \right\|_{\scriptscriptstyle \Sigma} \ll t^k \|g\|_{\scriptscriptstyle \Sigma} 
  \end{equation}
where $\|\cdot\|_{\scriptscriptstyle \Sigma}=D_{\scriptscriptstyle \Sigma}(\cdot)+\|\cdot\|_\infty$, the sum of the H\"older norm and the supremum norm over the shift. Note that $g_o$ is uniformly bounded from above and away from $0$ as $\Sigma$ is compact.
\begin{remark}
If $(i,\omega) \mapsto \lambda_{i,\omega}$ is constant, then $m=\rho$. Moreover if $\rho$ is invariant then $g_o=1$. If $\rho$ is a Bernoulli measure
then $\cA_k=(\cA_1)^k$ for every $k\ge 1$.
In this case annealed transfer operators were studied in \cite{Baladi}.
Note that $\cA_l \circ \cA_k = \cA_{l+k}$ if and only if $\rho$ is Bernoulli.
Averaged transfer operators were also considered in \cite{CRV17} in the special case that $\rho$ is a Bernoulli measure and all
potentials $\varphi_i$ are equal.
\end{remark}

\begin{remark}
The associated skew product 
\[F: X \times \Sigma \to X \times \Sigma, \quad (x, i_1 i_2 \ldots) \mapsto  (T_{i_1}(x), i_2i_3 \ldots)\]
reflects the time evolution along a given path in $\Sigma$ with a distribution on the space of possible paths, that is, the probability of the event of applying $T \in \mathcal{S}$ in time $n$   is given by 
$\rho(\{ \omega \in \Sigma : F^n(\, \cdot \,,\omega) =(T(\,\cdot\,), \theta^n(\omega)) \})$.
\end{remark}

We proceed to prove that the family $\{\mathcal A_n\}$ has exponential decay of correlations.
 Fix $k_0\in\mathbb N$ and $s\in (0,1)$ as given in Theorem \ref{theorem:contraction}. With $m$ defined as in \eqref{eq:measurem}, let $\pi\in\mathcal M_1(X)$ be given by $$d\pi : = d\mu_\omega dm(\omega).$$ For $f\in\mathcal H_\alpha$, let $$\|f\|_m:=\|\mu_{\cdot}(|f|)\|_\infty$$ being the supremum norm with respect to $m$ of the map $\omega\mapsto \mu_\omega(|f|)$ over the shift. 
 
\begin{theorem}\label{theo:annealed-conformal} Suppose the Ruelle-expanding semigroup $\mathcal{S}$ is jointly topologically mixing and finitely aperiodic, and that every potential $\varphi_i$ is $\alpha$-H\"older and summable. Suppose that every $\log p_i$, $i\in\cW$  is H\"older continuous on $\Sigma$. Then there exists $r\in(0,1)$ such that for all $f \in \mathcal{H}_\alpha$ and $n \geq 2k_0$
\[  \left| \frac{\mathcal{A}_n(f)(x)}{\mathcal{A}_n(\mathbf{1})(x)}  - \int f d\pi \right| \ll r^n (\overline D(f) +\|f\|_m).\]
Moreover, there exists a positive function $h\in\cH_\alpha$  such that for all $f\in\cH_\alpha$ and $n\ge 2k_0$,
$$\left|\frac{\mathcal A_n(f)(x)}{\beta^{n} h(x)}-\int f d\pi\right|\ll r^n(\overline D(f)+\|f\|_m),$$
with $\beta>0$ given by \eqref{eq:measurem}.
\end{theorem} 

\begin{proof} In the first step of the proof, we derive the first decay. 
Proposition \ref{prop:conformal} implies that for any $n\geq 2k_0, \omega\in\Sigma$ and $x\in X, f \in \mathcal{H}_\alpha$,
\[ \left| L_{[\omega]_n}(f)(x) -  \mu_{\omega}(f)  L_{[\omega]_n}(\mathbf{1})(x)  \right|   \ll s^n  \overline{D}(f)  L_{[\omega]_n}(\mathbf{1})(x) .\]
After integration, it yields that 
\begin{equation} \label{eq:estimate-0}  \left| \mathcal{A}_n(f)(x)  - \int \mu_{\omega}(f) L_{[\omega]_n}(\mathbf{1})(x) d\rho(\omega) \right|  \ll s^n \overline{D}(f) \mathcal{A}_n(\mathbf{1})(x).  \end{equation}
It remains to analyse  $\int \mu_{\omega}(f) L_{[\omega]_n}(\mathbf{1}) d\rho(\omega)$ as $n \to \infty$.
In order to do so, write $n=k+l$ with $l=[n/2]+1$.  Observe that by \eqref{eq:estimate-lambda} 
\begin{equation} \label{eq:estimate-nkl}
\left|L_{[\omega]_n}(\mathbf{1}) - \lambda_{[\omega]_k, \theta^k\omega} L_{[\theta^k\omega]_l}(\mathbf{1}) \right| \ll s^l \lambda_{[\omega]_k, \theta^k\omega} L_{[\theta^k\omega]_l}(\mathbf{1}).
\end{equation}
Note that it follows from  Proposition \ref{prop:conformal} that
 $\omega \mapsto \mu_{\omega}(f)$ is H\"older continuous on $\Sigma$ and its H\"older coefficient is bounded by a constant times $\overline{D}(f)$.
Hence
\begin{align*} 
& \left| \int \mu_{\omega}(f) L_{[\omega]_n}(\mathbf{1}) d\rho(\omega) - \int \mu_{\omega}(f) \lambda_{[\omega]_k,\theta^k \omega} L_{[\theta^k\omega]_l}(\mathbf{1}) d\rho(\omega)\right| \\
 \ll &  s^l \int  \mu_{\omega}(|f|) \lambda_{[\omega]_k, \theta^k \omega} L_{[\theta^k \omega]_l}(\mathbf{1})   d\rho(\omega)\\
 \overset{\eqref{eq:iotadual}}{=}&s^l\int \iota^k (\mu_\omega(|f|))\cdot  L_{[\omega]_l}(\mathbf 1) d\rho(\omega)\\
 =&s^l\int \left(\beta^{-k}g_o^{-1}\iota^k (\mu_\omega(|f|))-\pi(|f|)+\pi(|f|)\right)\cdot \iota^k(g_o) L_{[\omega]_l}(\mathbf 1) d\rho(\omega)\\
\overset{\eqref{eq:iotadecay}}{\ll} &s^l(t^k(\overline D(f) + \|f\|_m)+\pi(|f|))\int  \iota^k(g_o)  L_{[\omega]_l}(\mathbf{1}) d\rho(\omega)\\
\overset{\eqref{eq:iotadual}}=&s^l(t^k(\overline D(f) + \|f\|_m)+\pi(|f|)) \int g_o \cdot \lambda_{[\omega]_k, \theta^k\omega} L_{[\theta^k\omega]_l}(\mathbf{1}) d\rho(\omega) \\
\overset{\eqref{eq:estimate-nkl}} \ll&  s^l  (t^k(\overline D(f) + \|f\|_m)+\pi(|f|))\int  L_{[\omega]_n}(\mathbf{1}) \cdot g_o d\rho(\omega) \\
\ll& s^l  (t^k\overline D(f)+\|f\|_m) \mathcal{A}_{n}(\mathbf{1}).
\end{align*} 
Observe that in the previous estimate we have also showed that 
\begin{equation}\label{eq:qo}
\int \iota^k(g_o)L_{[\omega]_l}(\mathbf 1)d\rho(\omega)\ll \mathcal A_n(\mathbf 1).
\end{equation}
Then one can extract $\pi(f)$ by
\begin{align*} 
& \left|  \int  \mu_{\omega}(f) \lambda_{[\omega]_k, \theta^k\omega} L_{[\theta^k\omega]_l}(\mathbf{1})  d\rho(\omega)-  \pi(f) \int \lambda_{[\omega]_k, \theta^k\omega} L_{[\theta^k\omega]_l}(\mathbf{1})  d\rho(\omega) \right| \\
\overset{\eqref{eq:iotadual}}=  &  \left| \int \iota^k(\mu_\omega (f))  L_{[\omega]_l}(\mathbf{1}) d\rho(\omega) -  \pi(f) \int \iota^k(1)  L_{[\omega]_l}(\mathbf{1}) d\rho(\omega)\right|\\ 
=& \left|\int \left((\beta^{-k}g_o^{-1}\iota^k(\mu_\omega(f))-\pi(f))-(\beta^{-k}g_o^{-1}\iota^k(1)-1)\pi(f)\right)\iota^k(g_o) L_{[\omega]_l}(\mathbf 1)d\rho(\omega)\right|\\
\overset{\eqref{eq:iotadecay}}\ll  & t^k(\overline D(f) + \|f\|_m) \int  \iota^k(g_o)  L_{[\omega]_l}(\mathbf{1}) d\rho(\omega) \ll  t^k(\overline D(f) + \|f\|_m) \mathcal{A}_{n}(\mathbf{1}).
\end{align*}
Finally \eqref{eq:estimate-nkl} induces that
\begin{equation*}
 \left| \pi(f)\int \lambda_{[\omega]_k, \theta^k\omega} L_{[\theta^k\omega]_l}(\mathbf{1})  d\rho(\omega)  -   \pi(f)  \mathcal{A}_{n}(\mathbf{1})  \right| 
 \ll   s^l  |\pi(f)|  \mathcal{A}_{n}(\mathbf{1}). 
\end{equation*}
Combining the above estimates, one obtains that 
$$\left|\int \mu_\omega(f)L_{[\omega]_n}(\mathbf 1)d\rho(\omega)-\pi(f)\mathcal A_n(\mathbf 1)\right|\ll (t^k\overline D(f)+t^k\|f\|_m+s^l\|f\|_m)\mathcal A_n(\mathbf 1).$$
The first statement now follows from \eqref{eq:estimate-0} with $r=\max\{\sqrt s, \sqrt[3]t\}$.

We now proceed with proving the existence of $h$. In order to do so, let
$$\tilde \cA_n(x):=\int L_{[\omega]_n}(\mathbf 1)(x)\cdot g_o(\om) \,d\rho(\omega).$$ 
We first show that $\tilde I_n(x):=\beta^{-n}\tilde \cA_n(x)$ converges uniformly and exponentially fast to a positive function $h(x)\in \cH_\alpha$.

It follows from \eqref{eq:estimate-nkl} that for any $n=k+l$ with $l\ge k_0$,
 $$L_{[\omega]_n}(\mathbf 1)\asymp \lambda_{[\omega]_k, \theta^k\omega} L_{[\theta^k\omega]_{l}}(\mathbf{1}),$$ 
so that
$$
\tilde \cA_n\asymp \int \lambda_{[\omega]_k, \theta^k\omega} L_{[\theta^k\omega]_{l}}(\mathbf{1})\cdot g_o d\rho \overset{\eqref{eq:iotadual}}=\int \iota^k(g_o) L_{[\omega]_{l}}(\mathbf 1) d\rho=\beta^k \tilde \cA_l,$$
hence $\tilde I_n\asymp \tilde I_l$, especially $\tilde I_n\asymp \tilde I_{k_0}$ for all $n\ge k_0$. Since \eqref{eq:estimate-nkl} also implies that  
\begin{align*}
|\tilde \cA_n- \beta^k \tilde \cA_l|\ll s^{l} \beta^k\tilde \cA_l,
\end{align*}
one has $$|\tilde I_n-\tilde I_l|\ll s^l \tilde I_l.$$ Hence, $\{\tilde I_n(\cdot)\}$ is a Cauchy sequence. 
Denote the limit of $\tilde I_n(x)$ by $h(x)$. $\tilde I_n(x)$ converges uniformly to $h(x)$ since for $n\ge l\ge k_0$
$$|\tilde I_n-\tilde I_l|\ll s^l\tilde I_{k_0}\ll s^l.$$ Then because $\tilde I_n$'s are all H\"older, $h$ is H\"older as well. That $h$ is positive and $\|h\|_\infty$ is finite can be seen from $h\asymp\tilde I_{k_0}$. To see that the rate of convergence is exponential, for $n\geq k_0$ choose $j\in\mathbb N$ such that $|\tilde I_{jn}- h|\le s^n,$ then 
$$|\tilde I_n -h|\le |\tilde I_n-\tilde I_{2n}|+\cdots +|\tilde I_{(j-1)n}- \tilde I_{jn}|+|\tilde I_{jn}-h|\ll s^n.$$ 
Moreover,  Lemma \ref{comparability} infers that $\inf_{x\in X} \tilde I_{k_0}(x)>0$ and so are $\tilde I_n$ for $n\ge k_0$ and so is $h$. It follows that $\tilde I_n/h$ converges to $1$ uniformly and exponentially fast.

Next we show that $I_n(x):=\beta^{-n}\cA_n(\mathbf 1)(x)$ also tends to $h(x)$. 
For $n=k+l$ with $l\ge k_0$, because 
$$\left|\cA_n(\mathbf 1)-\int \iota^k(1)L_{[\omega]_l}(\mathbf 1) d\rho\right|\ll s^l\int \iota^k(1)L_{[\omega]_l}(\mathbf 1) d\rho$$
obtained from integrating \eqref{eq:estimate-nkl} and because 
\begin{align*}\left|\int (\iota^k(1)-\iota^k(g_o))L_{[\omega]_l}(\mathbf 1)d\rho\right|&=\left|\int (\beta^{-k}g_o^{-1}\iota^k(1)-1)\iota^k(g_o)L_{[\omega]_l} d\rho\right|\\
&\overset{\eqref{eq:iotadecay}}\le t^k\int \iota^k(g_o)L_{[\omega]_l} d\rho = t^k\beta^k \tilde A_l,
\end{align*}
one can deduce that 
$$|\cA_n(\mathbf 1)-\beta^k\tilde \cA_l|\ll (s^l+t^k)\beta^k\tilde A_l,$$
hence $$|{I_n}- \tilde I_l|\ll (s^l+t^k)\tilde I_l.$$
So that 
$$|I_n-h|\ll (s^l+t^k) h.$$
Lastly, applying Theorem \ref{theo:annealed-conformal}, one has that for all $f\in\cH_\alpha$ and $n\ge 2k_0$
\begin{align*}
|\beta^{-n}\cA_n(f)-\pi(f)h|&\le \beta^{-n}|\cA_n(f)-\pi(f)\cA_n(\mathbf 1)|+ \pi(f)|\beta^{-n}\cA_n(\mathbf 1)- h|\\
&\ll r^n(\overline D(f) +\|f\|_m) I_n+ \pi(f)|I_n-h|\\
&\ll r^n(\overline D(f)+ \|f\|_m) h.
\end{align*}
The second assertion on the decay  follows from this.
\end{proof}

The next result reveals an annealed version of decay of correlations.  

\begin{theorem} 
\label{theo:annealed decay of correlations}
Now suppose that the assumptions of the above theorem hold and that, in addition, $\rho$ is $\theta$-invariant.  Then there exist a probability measure $\tilde{\pi}$ on $\Sigma\times X$, $r \in (0,1)$ and $k_1 \in \mathbb{N}$ such that
\begin{align*} 
&\left| \int \sum_{v \in \cW^n} \mathbf{1}_{[v]}(\omega) f (T_v(x)) g(x) d\mu_\omega(x) d\rho(\omega) - 
\int f d\tilde{\pi} \int g d\mu_\omega d\rho \right| \\
& \leq   r^n \int |f| d\mu_\omega d\rho \left(  \overline D(g) + \int |g| d\mu_\omega d\rho \right)
\end{align*}
for all $g \in \cH_\alpha$ and $f: X \to \mathbb{R}$ integrable with respect to $d\mu_\omega(x) d\rho(\omega)$.
\end{theorem}

\begin{proof} For $\omega = (\omega_1 \omega_2 \ldots) \in \Sigma$, set $\lambda_{n,\omega} := \lambda_{\omega_1 \ldots \omega_n, \theta^n\omega}$ and $h_{n,\omega} := h_{\omega_1 \ldots \omega_n, \theta^n\omega}$, where $\lambda_\cdot$ and $h_\cdot$ are given by Proposition \ref{prop:conformal}. Moreover, Proposition \ref{prop:conformal} and Lemma \ref{comparability} imply for $n$ sufficiently large that  
\begin{align}
\nonumber 
& \int \sum_{v \in \cW^n} \mathbf{1}_{[v]}  f \circ T_v  g  d\mu_\omega  d\rho =  \int \sum_{v \in \cW^n} \mathbf{1}_{[v]}  f   \frac{L_v(g)}{\lambda_{n,\omega}}    d\mu_{\theta^n\omega}  d\rho  \\
\nonumber
= \;&  \int \sum_{v} \mathbf{1}_{[v]}  f  
\mu_\omega(g) \frac{L_v(\mathbf{1})}{\lambda_{n,\omega}}    d\mu_{\theta^n\omega}  d\rho \pm 2 s^n \overline D(g)  \int \sum_{v} \mathbf{1}_{[v]}  |f| \frac{L_v(\mathbf{1})}{\lambda_{n,\omega}}  d\mu_{\theta^n\omega} d\rho\\
\nonumber
 = \;&  \int \sum_{v \in \cW^n} \mathbf{1}_{[v]}  f  \mu_\omega(g) h_{n,\omega}  d\mu_{\theta^n\omega}  d\rho \pm C s^n \overline D(g)  \int \sum_{v \in \cW^n} \mathbf{1}_{[v]}  |f|  d\mu_{\theta^n\omega}  d\rho \\
 \label{eq:step 1 in annealed decay} 
 = \;& \int  f  \mu_\omega(g) h_{n,\omega}  d\mu_{\theta^n\omega}  d\rho \pm C s^n \overline D(g)  \int |f|  d\mu_{\omega}  d\rho, 
\end{align}
where $C/2$ is given by Lemma \ref{comparability}, and the last equality follows from $\theta$-invariance of $\rho$. Now assume that $n$ is even and  $n = 2m$. Then, by \ref{propitem:continnuity of h} of Proposition \ref{prop:equilibrium} 
there exists $C$ such that  
\begin{align}
\nonumber 
& \int  f  \mu_\omega(g) h_{n,\omega}  d\mu_{\theta^n\omega}  d\rho \\
\nonumber 
= \;&  \int  f  \mu_\omega(g) h_{m,\theta^m\omega}  d\mu_{\theta^n\omega}  d\rho \pm C s^m \int \mu_{\theta^n\omega}{(|f|)}  |\mu_\omega(g)|  d\rho 
\end{align}  
However, as $\omega \to \mu_\omega(g)$ is Lipschitz continuous by Proposition \ref{prop:conformal}, the exponential decay of correlations, say with rate $t \in (0,1)$ and the same constant $C>0$, applied to the error term implies that 
\begin{align} 
\nonumber 
&  \int  f  \mu_\omega(g) h_{m,\theta^m\omega}  d\mu_{\theta^n\omega}  d\rho \pm C s^m \int \mu_{\theta^n\omega}{(|f|)}  |\mu_\omega(g)|  d\rho \\
\label{eq:step 2 in annealed decay} 
= \; &  \int  f  \mu_\omega(g) h_{m,\theta^m\omega}  d\mu_{\theta^n\omega}  d\rho \pm C^2 s^m \int  \mu_{\omega}(|f|) d\rho \int \mu_\omega(|g|)   d\rho
\end{align}  
A further application of invariance and the exponential decay of correlations of $\theta$ to the main term and Lemma \ref{comparability}   gives that 
 \begin{align}
\nonumber 
& \int  f  \mu_\omega(g) h_{m,\omega}  d\mu_{\theta^n\omega}  d\rho = \int  \mu_\omega(g) \mu_{\theta^{2m}\omega} ( f \,h_{m,\theta^m\omega}  )  d\rho  \\
\label{eq:step 3 in annealed decay} 
= \;&  \int  \mu_\omega(g)d\rho \int f  h_{m,\omega} d\mu_{\theta^{m}\omega} d\rho
\pm C^2 t^m  \int  \mu_{\omega}(|f|) d\rho  \overline D(g)
\end{align}  
It hence remains to analyse $\int f  h_{m,\omega} d\mu_{\theta^{m}\omega}$. In order to do so, let $(\hat\Sigma,\hat\theta,\hat\rho)$ refer to natural extension of $\theta$. Then, again by  \ref{propitem:continnuity of h} of Proposition \ref{prop:equilibrium}, it follows that 
 \begin{align}
\nonumber 
&  \int f  h_{m,\omega} d\mu_{\theta^{m}\omega} d\rho(\omega)  
 = \int f  h_{m,\omega} d\mu_{{\theta}^{m}\omega} d\hat\rho (\tilde{\omega},\omega)\\
\nonumber
  = \;&
  \int f  h_{\tilde{\omega}_{-m}\cdots \tilde{\omega}_{-1},\omega} d\mu_{\omega} d\hat\rho (\tilde{\omega},\omega) 
 =   \int f  h_{\tilde{\omega},\omega} d\mu_{\omega} d\hat\rho (\tilde{\omega},\omega)
  \pm Cs^m \int  \mu_{\omega}(|f|) d\rho \\
\label{eq:step 4 in annealed decay}   
  = \;& \int f   d\mu_{\tilde{\omega},\omega} d\hat\rho (\tilde{\omega},\omega)
  \pm Cs^m \int  \mu_{\omega}(|f|) d\rho.
\end{align}
Let $d\tilde{\pi}(x) := d\mu_{\tilde{\omega},\omega}(x) d\hat\rho (\tilde{\omega},\omega)$. The
 theorem now  follows by combining \eqref{eq:step 1 in annealed decay}, \eqref{eq:step 2 in annealed decay}, \eqref{eq:step 3 in annealed decay} and \eqref{eq:step 4 in annealed decay}.
\end{proof}

\begin{remark} As a corollary of the proof, we also obtain  an explicit representation of $\tilde{\pi}$. That is, $d\tilde{\pi}(x) := d\mu_{\tilde{\omega},\omega}(x) d\hat\rho (\tilde{\omega},\omega)$ where $\hat\rho$ is the natural extension of $\rho$
(which is assumed invariant). 
In particular, $d\tilde{\pi}$ and  $d\mu_{\omega} d\rho(\omega)$ are equivalent measures, even though $d\tilde{\pi}/d\mu_{\omega} d\rho(\omega)$ might be a function depending on $\omega$. However, it is not clear if $\tilde{\pi}$ and $\pi$ coincide. Furthermore, this representation reveals that in our sequential setting, the measure arising in the annealed version of decay of correlations is   
 an integral of the path-wise equilibrium measures, as known for the special case where $\rho$ is a Bernoulli measure.  
\end{remark}

\section{An almost sure invariance principle}
Exponential decay has many implications on the statistical behaviour of the dynamical system. 
For sequential dynamical systems of expanding maps of the interval, first versions of central limit theorems were obtained by Heinrich and Conze \& Raugi (\cite{Heinrich:1996,ConzeRaugi:2007}). We now show an almost sure invariance principle in the setting of  Ruelle expanding maps.
It is worth mentioning that almost sure invariance principles have been obtained in the context of
quenched random dynamical systems (see e.g. \cite{Davor} and references therein).
 Let $\mathcal B$ be the Borel $\sigma$-algebra on $X$. 
With respect to the measure $\mu_{uv\omega}$, where $u,v $ are finite words and $\omega$ is an infinite word, $\mathbb P_u^v$ can be seen as a conditional expectation in the following way.
\begin{lemma}\label{lem:condexptn} For any $f\in\cH_\alpha$ $$\mathbb E_{\mu_{uv\omega}}(f\circ T_u|T^{-1}_{uv}\mathcal B)=\mathbb P_u^v(f)\circ T_{uv}.$$
\end{lemma}
\begin{proof}
For any $A\in\mathcal B$, using (iii) of Proposition \ref{prop:conformal}, 
\begin{align*}
\int_{T^{-1}_{uv}A} f\circ T_u d\mu_{uv\omega}&=\int \mathbf 1_A\circ T_v\cdot f d\mu_{uv\omega}\circ T^{-1}_u=\int \mathbf 1_A\circ T_v\cdot f d\mu_{u, v\omega}\\
&=\int \mathbf 1_A\circ T_v\cdot f d{\mathbb P_u^v}^\ast(\mu_{uv,\omega})=\int \mathbb P_u^v(\mathbf 1_A\circ T_v\cdot f) d\mu_{uv,\omega}\\
&=\int \mathbf 1_A\cdot \mathbb P_u^v(f)d\mu_{uv, \omega}=\int_A \mathbb P_u^v(f) d\mu_{uv\omega}\circ T^{-1}_{uv}\\
&=\int_{T_{uv}^{-1}A}\mathbb P_u^v(f)\circ T_{uv}d\mu_{uv\omega}.
\end{align*}
\end{proof}

The almost sure invariance principle we are going to show is similar to the one in \cite{StadlbauerZhang:2017a} for non-stationary shift. Both are based on the almost sure invariance principle for reverse martingale differences by Cuny and Merlev\`ede.

\begin{theorem}[{\cite[Theorem 2.3]{CunyMerlevede:2015}}]\label{thm:rvsmtg}
Let $(U_n)_{n\in\mathbb N}$ be a sequence of square integrable reverse martingale differences with respect to a non-increasing filtration $(\mathcal G_n)_{n\in\mathbb N}$. Assume that $\sigma_n^2:=\sum_{k=1}^n\mathbb E(U_k^2)\to\infty$ and that $\sup_n\mathbb E(U_n^2)<\infty$. Assume that
\begin{align*}
&\sum_{k=1}^n\left(\mathbb E(U_k^2|\mathcal G_{k+1})-\mathbb E(U^2_k)\right)=o(\sigma_n^2) \qquad {a.s.}\\
&\sum_{n\geq 1}\sigma_n^{-2t}\mathbb E(|U_n|^{2t})<\infty \qquad \text{for some } 1\leq t\leq 2.\label{eq:sumvar}
\end{align*}
Then, enlarging our probability space if necessary, it is possible to find a sequence $(Z_k)_{k\geq 1}$ of independent centred Gaussian variables with $\mathbb E(Z_k^2)=\mathbb E(U_k^2)$ such that
$$\sup_{1\leq k\leq n}|\sum_{i=1}^k U_i-\sum_{i=1}^k Z_i|=o(\sqrt{\sigma_n^2\log\log \sigma_n^2})\qquad {a.s.}$$
\end{theorem}

We need to make another assumption.
\begin{definition} A $(a,\lambda)$-Ruelle-expanding map $T$ is\emph{finitely expanding} if 
$$\sup_{\substack{x,y\in X \\ 0<d(x,y)<a}}\frac{d(T(x), T(y))}{d(x,y)}<\infty.$$
We refer to $\mathcal S$ as finitely Ruelle-expanding if every $T_i, i\in \cW$ satisfies this property.
\end{definition}

\begin{theorem} \label{thm:asip}
Suppose the finitely Ruelle-expanding semigroup $\mathcal{S}$ is jointly topologically mixing and finitely aperiodic, and that every potential $\varphi_i$ is $\alpha$-H\"older and summable. 
Suppose $\omega\in\Sigma$, $f\in\cH_\alpha$. Let $f_n=f -\int f\circ T_{[\omega]_n} d\mu_{\omega}$ for every $n\in\mathbb N_0$ and let $s_n^2 = \mathbb E_{\mu_\omega}(\sum_{k=0}^{n-1} f_k\circ T_{[\omega]_k})^2$ for $n\ge 1$. Assume that 
$$\quad \sum_n s_n^{-4}<\infty.$$ Then, enlarging our probability space if necessary, there exists a sequence $(Z_n)$ of independent centred Gaussian random variables such that 
\begin{gather*}
\sup_n\left|\sqrt{\textstyle \sum_{k=0}^{n-1} \mathbb E_{\mu_\omega} Z_k^2}-s_n\right|<\infty,\\
\sup_{0\leq k \leq n-1} \left| \textstyle\sum_{i=0}^k f_i\circ T_{[\omega]_i} - \sum_{i=0}^k Z_i \right| = o(\sqrt{s^2_n \log \log s^2_n})\quad \mu_\omega{\rm-a.s.}.
\end{gather*}
\end{theorem}
\begin{proof}
Denote $\mathcal B_n=T_{[\omega]_n}^{-1}\mathcal B$ for $n\in\mathbb N$ and let $\mathcal B_0=\mathcal B$, then $\mathcal B_n$ is a non-increasing filtration. Let $h_0=0$ and define $h_n\in \mathcal{H}_\alpha$ recursively by $h_{n+1}=\mathbb P_{[\omega]_n}^{[\theta^n\omega]_1}(f_n+h_n)$. Then \eqref{eq:iteration_1} implies that $h_n=\sum_{k=0}^{n-1}\mathbb P_{[\omega]_k}^{[\theta^k\omega]_{n-k}}f_k\in\cH_\alpha$. It follows from Proposition \ref{prop:conformal} that $\mu_\omega\circ T^{-1}_{[\omega]_k}=\mu_{[\omega]_k, \theta^k\omega}$, then
$$\mathbb P_{[\omega]_k}^{[\theta^k\omega]_{n-k}}f_k=\mathbb P_{[\omega]_k}^{[\theta^k\omega]_{n-k}}f-\int f\circ T_{[\omega]_k}d\mu_\omega=\mathbb P_{[\omega]_k}^{[\theta^k\omega]_{n-k}}f-\int f d\mu_{[\omega]_k,\theta^k\omega}$$
and that, with $k_0\in\mathbb N$ and $s\in(0,1)$ given by Theorem \ref{theorem:contraction}
\begin{align*}
\|h_n\|&\le\sum_{k=0}^{n-k_0}2s^{n-k}\overline D(f)+\sum_{k=n-k_0+1}^{n-1}\|\mathbb P_{[\omega]_k}^{[\theta^k\omega]_{n-k}}f_k\|\\
&\le \sum_{k=0}^{n-k_0}2s^{n-k}\overline D(f)+\sum_{k=n-k_0+1}^{n-1}C\|f\|\ll \|f\|,
\end{align*}
where $C$ is a uniform bound for all $\|\mathbb P_u^v\|$ (Lemma \ref{lem:doeblin-fortet}).

Let $$U_n:=f_n\circ T_{[\omega]_n}+h_n\circ T_{[\omega]_n}-h_{n+1}\circ T_{[\omega]_{n+1}}.$$ $U_n$ is $\mathcal B_n$-measurable and square integrable. Moreover, apply Lemma \ref{lem:condexptn} to get that
\begin{align*}\mathbb E_{\mu_\omega}(U_n|\mathcal B_{n+1})=\mathbb P_{[\omega]_n}^{[\theta^n\omega]_1} f_n\circ T_{[\omega]_{n+1}}+\mathbb P_{[\omega]_n}^{[\theta^n\omega]_1} h_n\circ T_{[\omega]_{n+1}}-h_{n+1}\circ T_{[\omega]_{n+1}}=0.
\end{align*} 
So $(U_n)_{n\in\mathbb N_0}$ is a sequence of square integrable reverse martingale differences. Let $$\sigma_n^2:=\sum_{k=0}^{n-1}\mathbb E_{\mu_\omega}U_k^2=\mathbb E_{\mu_\omega}\left(\sum_{k=0}^{n-1}U_k\right)^2.$$ We check the conditions of Theorem \ref{thm:rvsmtg}. $\mathbb E$ in the rest of the proof stands for $\mathbb E_{\mu_\omega}$.

First we show $\sigma_n^2\to\infty$ and $\sup_n\mathbb EU_n^2<\infty$. It follows from 
\begin{align*}
|\sigma_n-s_n|&=\left|\mathbb E^{1/2}(\sum_{k=0}^{n-1}U_k)^2-\mathbb E^{1/2}(\sum_{k=0}^{n-1}f_k\circ T_{[\omega]_k})^2\right|\\
&\leq \mathbb E^{1/2}(\sum_{k=0}^{n-1}U_k-\sum_{k=0}^{n-1}f_k\circ T_0^k)^2=\mathbb E^{1/2}(h_n\circ T_{[\omega]_n})^2\\
&\ll \|f\|
\end{align*}
that $|\sigma_n-s_n|$ is uniformly bounded. So $s_n^2\to\infty$ implies that $\sigma_n^2\to\infty$. Since $\|U_n\|_\infty$ is uniformly bounded, $\sup_n\mathbb E U_n^2<\infty$.

Next we show that $$\sum_{k=0}^{n-1}\left(\mathbb E(U_k^2|\mathcal B_{k+1})-\mathbb E(U^2_k)\right)=o(\sigma^2_n) \qquad \mathbb \mu_\omega\text{-a.s.}$$
Let $u_n=f_n+h_n-h_{n+1}\circ T_{[\theta^n \omega]_1}$ and let $\tilde u_n=u_n^2-\mathbb E U_n^2$ . Then $\|\tilde u_n\|_\infty\ll \|f\|^2.$ Moreover, the H\"older coefficient of $\tilde u_n$ is also uniformly bounded because, denoting $[\theta^{n-1}\omega]_1=i\in\cW$,
\begin{align*}
D_\alpha(h_{n}\circ T_i)&=\sup_{x\neq y\in X}\frac{|h_{n}\circ T_i (x)-h_{n}\circ T_i(y)|}{d(x, y)^\alpha}\\
&\le  D_\alpha(h_n)\cdot \sup_{0<d(x,y)<a}\left(\frac{d(T_i(x), T_i(y))}{d(x,y)}\right)^\alpha+2a^{-\alpha}\|h_n\|_\infty
\end{align*}
which is uniformly bounded by assumption.
Let $$F_n =\sigma_n^{-2}\sum_{k=0}^{n-1}\mathbb E(U_k^2|\mathcal B_{k+1}),$$ 
then $$\sum_{k=0}^{n-1}\left(\mathbb E(U_k^2|\mathcal B_{k+1})-\mathbb E(U_k^2)\right)=\sum_{k=0}^{n-1}\mathbb P_{[\omega]_k}^{[\theta^k\omega]_1} \tilde u_{k}\circ T_{[\omega]_{k+1}}=\sigma_n^2(F_n-1).$$ 
Applying Proposition \ref{prop:conformal}, we have 
\begin{align*}
\mathbb E\left(\sum_{k=0}^{n-1} \mathbb P_{[\omega]_k}^{[\theta^k\omega]_1}\tilde u_k\circ T_{[\omega]_{k+1}}\right)^2
&\ll \sum_{0\leq k\leq l\leq n-1}\mathbb E \left(\mathbb P_{[\omega]_k}^{[\theta^k\omega]_1}\tilde u_k\circ T_{[\omega]_{k+1}}\cdot \mathbb P_{[\omega]_l}^{[\theta^l\omega]_1}\tilde u_l\circ T_{[\omega]_{l+1}}\right)\\
&=\sum_{0\leq k\leq l\leq n-1}\int \mathbb P_{[\omega]_k}^{[\theta^k\omega]_{l-k+1}}\tilde u_k\cdot \mathbb P_{[\omega]_l}^{[\theta^l\omega]_1}\tilde u_l ~ d\mu_{[\omega]_{l+1},\theta^{l+1}\omega}\\
&\ll \sum_{l-k+1\ge k_0}s^{l-k+1} \overline D\tilde u_k\cdot \mathbb EU_l^2 +\sum_{l-k+1<k_0}\| \tilde u_k\|_\infty\cdot \mathbb EU_l^2\\
&\ll k_0\cdot \sum_{l=0}^{k_0-2} \mathbb EU_l^2+(s^{k_0}+k_0)\cdot \sum_{l=k_0-1}^{n-1}\mathbb EU_l^2
\end{align*}
where in the last inequality we have used that $\|\tilde u_k\|$ is uniformly bounded.
Therefore
$$\mathbb E(F_n-1)^2=\sigma_n^{-4}\mathbb E\left(\sum_{k=0}^{n-1} \mathbb P_{[\omega]_k}^{[\theta^k\omega]_1} \tilde u_{k}\circ T_{[\omega]_{k+1}}\right)^2\ll \sigma_n^{-4}\sum_{l=0}^{n-1}\mathbb E U_l^2= \sigma_n^{-2}.$$
As $\sigma_n\to\infty$, $\mathbb E(F_n-1)^2\to 0$. We need to show that it is almost sure convergence. Let $ C = \sup_n \mathbb E U^2_n$ and  let $k_n=\inf \{k: \sigma_k^2\geq n^2 C\}.$
Then $k_n<\infty, k_n\to\infty$ and $$n^2C\leq\sigma_{k_n}^2\leq(n^2+1)C.$$
Since $$\sum_n \mathbb E(F_{k_n}-1)^2\ll\sum_n\sigma_{k_n}^{-2}<\infty,$$ $F_{k_n}\to 1$ a.s. by the Borel-Cantelli lemma.  Let $m=m(n)\to\infty$ be such that $k_{m}\leq n\leq k_{m+1}$, then 
$$ F_{k_m}\frac{m^2}{(m+1)^2+1}\leq F_{k_m}\frac{\sigma^2_{k_m}}{\sigma^2_{k_{m+1}}}\leq F_n\leq  F_{k_{m+1}}\frac{\sigma^2_{k_{m+1}}}{\sigma_{k_{m}}^2}\leq F_{k_{m+1}}\frac{(m+1)^2+1}{m^2}.$$
Hence, $F_n\to 1$ a.s.
Lastly, $\sum_{n}\sigma_n^{-2}\mathbb E U_n^{2}<\infty$ because $\|U_n\|_\infty$ is uniformly bounded, $|\sigma_n-s_n|\ll \|f\|$ and $\sum_n s_n^{-4}<\infty$ by assumption.

Now we can use Theorem \ref{thm:rvsmtg} to find a sequence of independent centred Gaussian variables $\{Z_k\}$ with $\mathbb EZ_k^2=\mathbb EU_k^2$ such that 
$$\sup_{0\leq k\leq n-1}\left|\sum_{i=0}^k U_i-\sum_{i=0}^k Z_i\right|=o\left(\sqrt{\sigma^2_n \log\log \sigma^2_n}\right)\qquad \text{a.s.}$$
Since $|\sum_{i=0}^{k} f_i\circ T_{[\omega]_i}-\sum_{i=0}^{k} U_i|$ and $|\sigma_n-s_n|$ are both uniformly bounded, the statement of the theorem  follows.
\end{proof}

\section{Applications} \label{sec:app} 

In this section we illustrate some possible applications of our main results, both for conformal iterated function systems and the thermodynamic formalism of 
free semigroup actions by expanding maps.

\subsection{Non-autonomous conformal iterated function systems} The class of  
non-autonomous conformal iterated function system was introduced and studied in \cite{Rempe-GillenUrbanski:2016}, and is defined as follows.

\begin{definition} \label{def:NCIFS} 
We refer to $\{X,(\Phi_i:1\leq i \leq k)\}$ as a \emph{non-autonomous conformal iterated function system} if 
$X$ is a convex, compact subset of $\mathbb{R}^d$, for some $d \in\mathbb{N}$, 
with  $\overline{\hbox{int}(X)} =X$, and $(\Phi_i)$ 
is a collection $\{ \varphi_{i,1},\ldots,\varphi_{i,k(i)}\}$ of maps from $X$ to $X$
such that
\begin{enumerate}
\item \label{number:conformal} the following \emph{conformality condition} holds:  there exists an open connected set $V \supset X$ such that each $\varphi_{i,j}$ extends to a continuously differentiable conformal diffeomorphism from $V$ into $V$,
\item the \emph{open set condition} holds: $\varphi_{i,j}(\hbox{int}(X))  \cap \varphi_{i,\tilde{j}}(\hbox{int}(X)) = \emptyset$, for all $1 \leq j < \tilde{j} \leq k(i)$ and $i= 1,\ldots k$,
\item \label{number:hoelder} the following conditions on \emph{bounded distortion and uniform contraction} hold: there exist constants $K \geq	1$ and $\eta \in (0,1)$ such that for any  $n \in\mathbb{N}$ and any choice 
$(i_1,j_1),\ldots, (i_n,j_n)$, with $i_l \in \{1,\ldots, k\}$  and $1\leq j_l \leq k(l)$ and all $x,y \in X$, for 
$\varphi:= \varphi_{i_n,j_n} \circ \cdots \circ \varphi_{i_1,j_1}$, we have that 
\[ \|D \varphi(x)\| \leq K \|D \varphi(y)\|, \quad \|D \varphi(x)\| \leq K \eta^n.   \]
\end{enumerate}
\end{definition}
As $X$ is assumed to be compact and the $k(i) < \infty$ for all $i= 1,\ldots k$, it follows for any compact set $A \subset K$ that $\Phi_i(A):= \cup_{j=1}^{k(i)} \varphi_{i,j}(A)$ is compact. Hence, for a given $\omega \in \Sigma$, where $\Sigma = \{(\omega_1\omega_2\ldots):  1 \leq \omega_i \leq k\}$,  $(\Phi_{\omega_1} \circ   \cdots \circ  \Phi_{\omega_n} (X))_n$ is a decreasing sequence of compact sets which then implies that the \emph{limit set} $J_\omega$, defined by 
 \[J_\omega := \lim_{n \to \infty}  \Phi_{\omega_1} \circ \Phi_{\omega_2} \circ \cdots \circ  \Phi_{\omega_n} (X) \]   
is non-empty and compact. 

We now derive an averaged version of Bowen's formula in order to have access to the Hausdorff dimension of these limit sets. In order to do so, we have to adapt the semigroup setting to the IFS. First observe that \eqref{number:conformal} in Definition \ref{def:NCIFS} implies that $\varphi:= \varphi_{i_n,j_n} \circ \cdots \circ \varphi_{i_1,j_1}$ is a well-defined  conformal diffeomorphism, for any $n \in\mathbb{N}$ and $(i_1,j_1),\ldots, (i_n,j_n)$, with $i_l \in \{1,\ldots, k\}$  and $1\leq j_l \leq k(l)$. Furthermore, by \eqref{number:hoelder}, $\varphi$ is a contraction with rate $K\eta^n$ and, by a  standard argument, $x \mapsto \log \|D \varphi(x)\|$ is Lipschitz continuous with respect to a uniform constant.

For $\delta \geq 0$, we now consider the operators, for $w=(\omega_1 \ldots \omega_n)$,   
\begin{align*} L^\delta_{\omega_i}(f) &:= \sum_{j=1}^{k(\omega_i)} \|D\varphi_{\omega_i,j} (\,\cdot\,)\|^\delta f\circ\varphi_{\omega_i,j}, \\
L^\delta_{w}(f) &:= \sum_{j_1, \ldots, j_n} \|D(\varphi_{\omega_1,j_1} \cdots \varphi_{\omega_n,j_n}) (\,\cdot\,)\|^\delta f \circ \varphi_{\omega_1,j_1} \cdots \varphi_{\omega_n,j_n} \\
&\phantom{:}= L^\delta_{\omega_1} \circ L^\delta_{\omega_2} \circ \cdots \circ L^\delta_{\omega_n} (f).  
\end{align*} 
for $f$ in a suitable function space (the last equality follows from conformality). Now assume that $\rho$ is a probability measure on $\Sigma$  which satisfies the conditions of Theorem \ref{theo:annealed-conformal}, that is $\log d\rho/d\rho\circ\sigma$ is H\"older continuous and the support of  $\rho$ is a topological mixing SFT, and, for $n \in \mathbb{N}$,
\[ \mathcal{A}^\delta_n := \sum_{w \in \{1,\ldots k\}^n} \rho([w]) L^\delta_w. \]
Here $[w]$ represents the cylinder set $\{\omega\in\Sigma: [\omega]_n=w\}$. Observe that the arguments in the proofs of Theorems~\ref{thm:mainA} and \ref{thm:mainB} apply straightforwardly in this context through an interpretation of $\varphi_{\omega_1,j_1} \cdots \varphi_{\omega_n,j_n}$ as inverse branch of an expanding map. Hence, we obtain uniform and exponential convergence of $L^\delta_w$ as $|w| \to \infty$ and of $\mathcal{A}^\delta_n$ as $n \to \infty$, respectively.  In particular, for each $\delta \geq 0$, there exists $\lambda_\delta$ such that $\mathcal{A}^\delta_n(\mathbf{1}) \asymp \lambda_\delta^n$. Thus, the annealed pressure function $P:[0,\infty) \to \mathbb{R}$ given by
\[ P(\delta) := \lim_{n \to \infty} \frac{1}{n} \log \mathcal{A}^\delta_n(\mathbf{1}) = \log \lambda_\delta\] 
is well defined.
\begin{lemma} The function $P$ is continuous and strictly decreasing. Furthermore, $\lim_{\delta \to +\infty} P(\delta) = -\infty$ and 
$P_0 = \log \lambda_0 \geq \log(\min_i k(i))$, where $\lambda_0$ is the spectral radius of the operator defined by 
\[\iota(f)  = \sum_{i=1}^{k} k(i) \frac{d\rho}{d\rho\circ\sigma}(i\,\cdot \,) f(i\,\cdot \,).\]
\end{lemma}
\begin{proof} It follows from the definition and the finiteness of the generating IFS that there exist $\eta_+,\eta_- \in (0,1)$ such that   $\eta_-^n \ll \|D(\varphi_{\omega_1,j_1} \cdots \varphi_{\omega_n,j_n})\| \ll \eta_+^n$. Hence, for $\epsilon > 0$, we have that 
\[ \eta_-^{n\epsilon}\mathcal{A}^\delta_n(\mathbf{1}) \ll \mathcal{A}^{\delta+\epsilon}_n(\mathbf{1}) \ll  \eta_+^{n\epsilon} \mathcal{A}^\delta_n(\mathbf{1}), \]
which implies that $ \epsilon \log \eta_- \leq P(\delta+\epsilon) -P(\delta)  \leq  \epsilon \log \eta_+ $. Hence,  $P$ is continuous and strictly decreasing. In order to determine  $\lim_{\delta \to +\infty} P(\delta) = -\infty$, observe that 
\begin{align*} \mathcal{A}^\delta_{m+n}(\mathbf{1})(x) &\leq \sum_{|v|=m} \sum_{|w|=n} \rho([vw]) L^\delta_v \circ L^\delta_w (\mathbf{1}) (x) \\ 
& \leq  \sum_{|v|=m} \rho([v]) L^\delta_v \left(  \sum_{|w|=n}  \frac{\rho([vw])}{\rho([v])\rho([w])} \rho([w]) L^\delta_w (\mathbf{1}) \right) (x) \\
& \leq  C \mathcal{A}^\delta_{m} \circ   \mathcal{A}^\delta_{n}(\mathbf{1})(x), \qquad \forall m,n \ge 1
\end{align*}
as there is a uniform bound $C$ for ${\rho([v])\rho([w])}/{\rho([vw])}$ by bounded distortion of $\rho$. Hence, for every fixed $n\ge 1$,
\[\lambda_\delta = \lim_{l} \sqrt[ln]{\mathcal{A}^\delta_{ln}(\mathbf{1})} \leq  
\sqrt[n]{ C\|\mathcal{A}^\delta_{n}(\mathbf{1})\|_\infty} \xrightarrow{\delta \to +\infty} 0.\]

In order to determine $P(0)$, we employ Theorem \ref{theo:annealed-conformal} as follows. For $\delta =0$, $L_i(\mathbf{1}) = k(i)\mathbf{1}$. Hence, by the proof of Theorem \ref{theo:annealed-conformal}, $\lambda_0$ is the spectral radius of $\iota$ which is bigger than or equal to $\log(\min_i k(i))$. 
\end{proof}

As an immediate corollary, it follows that there exists a unique $\delta_0 > 0$ such that $P(\delta_0)=0$, provided that $P(0) > 0$, e.g. if $\min_i k(i)> 1$.

\begin{theorem} Assume that $P(0)>0$. Then, for $\rho$-a.e. $\omega$, the Hausdorff dimension $\dim_H(J_\omega)$ of $J_\omega$ is equal to the unique root $\delta_0$ of $P$. 
\end{theorem}
 
\begin{proof} 
Fix $x\in X$. In analogy to the above pressure function, for $\omega = (\omega_i)$ set
\[ P_\omega(\delta) := \limsup_{n \to \infty} \frac{1}{n} \log L^\delta_{\omega_1\ldots \omega_n} (\mathbf{1}) (x).   \]
In order to prove almost sure convergence, we employ Kingman's subadditive ergodic theorem. In order to do so, observe that the shift is $\rho$-ergodic, and that there exists an equivalent invariant probability measure. Set 
\[ g_n(\omega) := \sup\left \{  \log L^\delta_{\omega_1\ldots \omega_n} (\mathbf{1}) (x):  {x \in X}\right\}.\]
 By construction, $g_{m+n}(\omega) \leq g_m(\omega) + g_n(\sigma^n(\omega))$. As $g_n(\omega) \asymp \log L^\delta_{\omega_1\ldots \omega_n} (\mathbf{1}) (x)$, it now follows from Kingman's subadditive ergodic theorem that $P_\omega(\delta)$ exists almost everywhere and in $L^1(\rho)$, that $P_\omega(\delta)$ is almost surely constant and that the $\limsup$ in the definition in fact is a limit. It follows from these observations that $P_\omega(\delta) = P(\delta)$ almost surely, but for $\delta$ fixed. However, by  the same argument for Lipschitz continuity of $P$ in the proof above, one obtains that the maps $P_\omega$ are equi-Lipschitz continuous. Hence, by choosing a countable and dense set $\{\delta_i\}$, one obtains a set of full measure $\Omega$ such that $P_\omega(\delta) = P(\delta)$ for all $\omega \in \Omega$ and $\delta \geq 0$. 
\smallskip 

We now show that $\dim_H(J_\omega) = \delta_0$ for each $\omega = (\omega_i) \in \Omega$. In order to do so, we first recall some consequences of conformality. As $\varphi:=
\varphi_{\omega_1,j_1} \cdots \varphi_{\omega_n,j_n}$ is conformal, it follows that the diameter $\hbox{diam}(\varphi(X))$ satisfies $\hbox{diam}(\varphi(X)) \asymp \|D\varphi\| \cdot \hbox{diam}(X)$. Furthermore, covers by sets of type $\varphi(X)$ are optimal in the following sense. By Lemma 2.7 in 
\cite{MauldinUrbanski:1996}, or from the proof of Theorem 3.2 in \cite{Rempe-GillenUrbanski:2016}, 
there exists $M\in \mathbb{N}$ such that, for each ball $B$ of radius $r>0$, there exist a subset $W(B)$ of $\{((\omega_1,j_1), \cdots (\omega_n,j_n)) : n \in \mathbb{N}, 1 \leq j_i \leq k(i)\}$ of at most $M$ elements such that 
\begin{enumerate}
\item the elements of $\{\varphi_{\omega_1,j_1} \cdots \varphi_{\omega_n,j_n} (\hbox{int}(X)) :((\omega_1,j_1), \ldots (\omega_n,j_n)) \in W(B) \}$ are pairwise disjoint,
\item $\hbox{diam}(\varphi_{\omega_1,j_1} \cdots \varphi_{\omega_n,j_n} (X)) \asymp \hbox{diam}(B)$  for all $((\omega_1,j_1), \ldots (\omega_n,j_n)) \in W(B)$,
\item $B \cap J_\omega \subset \bigcup_{((\omega_1,j_1), \ldots (\omega_n,j_n)) \in W(B)}  \varphi_{\omega_1,j_1} \cdots \varphi_{\omega_n,j_n} (X) $.
\end{enumerate}
The result now provides access to the $\delta$-Hausdorff measure of $J_\omega$ as follows. Assume that $\mathcal{U}$ is a finite cover of $J_\omega$ by closed balls. By replacing each $B \in \mathcal{U}$ by
$\{\varphi_{\omega_1,j_1} \cdots \varphi_{\omega_n,j_n} (X) :((\omega_1,j_1), \ldots (\omega_n,j_n)) \in W(B) \}$, we obtain a further cover $\mathcal{V}$ which satisfies 
\begin{align*}
\sum_{B \in \mathcal{U}}  \hbox{diam}(B)^\delta \asymp  \sum_{A \in \mathcal{V}}  \hbox{diam}(A)^\delta.
\end{align*}
Hence, in order to estimate the right hand side, we may assume without loss of generality that for each $B \in \mathcal{U}$, there exist $(\omega_i,j_i)$ such that $B = \varphi_{\omega_1,j_1} \cdots \varphi_{\omega_n,j_n} (X)$. On the other hand, Proposition \ref{prop:conformal} implies that for an arbitrary $x \in \hbox{int}(X)$, 
\begin{align*}
\mu_\omega(B) 
& = \lim_{l \to \infty} \frac{L^\delta_{\omega_{n+1} \ldots \omega_{n+l}}\circ L^\delta_{\omega_1 \ldots \omega_{n}}(\mathbf{1}_B)(x)}{L^\delta_{\omega_1 \ldots \omega_{n+l}}(\mathbf{1})(x)} \\
& \asymp \|D\varphi_{\omega_1,j_1} \cdots \varphi_{\omega_n,j_n}\|^\delta
\lim_{l \to \infty} \frac{L^\delta_{\omega_{n+1} \ldots \omega_{n+l}}(\mathbf{1})(x)}{L^\delta_{\omega_1 \ldots \omega_{n+l}}(\mathbf{1})(x)} \asymp \hbox{diam}(B)^\delta \lambda_{\omega_1 \ldots \omega_{n},\sigma^n\omega}^{-1} 
\end{align*}
Setting $|B|=n$, this implies that 
\begin{align*}
\sum_{B \in \mathcal{U}}  \hbox{diam}(B)^\delta \asymp  
\sum_{B \in \mathcal{U}}  \lambda_{\omega_1 \ldots \omega_{|B|},\sigma^{|B|}\omega} \mu_\omega(B). 
\end{align*}
Now assume that the interiors of the elements of $\mathcal{U}$ are disjoint. Then $\sum \mu_\omega(B) =1$ and the asymptotics of $\sum \hbox{diam}(B)^\delta $ as $\max \hbox{diam}(B) \to 0$ are determined by the asymptotics of $\lambda_{\omega_1 \ldots \omega_{n},\sigma^{n}\omega}$ as $n \to \infty$. Hence, if $\delta> \delta_0$, then the $\delta$-Hausdorff measure of $J_\omega$ is $0$ and if $\delta < \delta_0$, then the $\delta$-Hausdorff measure of $J_\omega$ is $\infty$. This implies that $\dim_H(J_\omega) = \delta_0$.
\end{proof}

\subsection{Thermodynamic formalism of semigroup actions}  

In this subsection we will provide some applications of our results to the setting of
finitely generated free semigroup actions.

Let $X$ be a compact metric space, $\varphi: X \to \mathbb R$ be a continuous potential and let 
$G_1=\{g_1, g_2, \dots, g_k\}$ be a finite set of continuous self maps on $X$, for some $k\geq 2$. 
The semigroup $\mathcal{S}$ generated by $G_1$ induces a continuous semigroup {action} given by
$$
\begin{array}{rccc}
\mathbb{S} : & \mathcal{S} \times X & \to & X \\
	& (g,x) & \mapsto & g(x),
\end{array}
$$
meaning that for any $\underline{g},\,\underline{h}  \in \mathcal{S}$ and every $x \in X$, we have $\mathbb{S}(\underline{g}\,\underline{h},x)=\mathbb{S}(\underline{g}, \mathbb{S}(\underline{h},x)).$
The thermodynamic formalism of semigroup actions faces several difficulties. On one hand, 
while probability measures which are invariant by all generators may fail to exist, 
in opposition to the case of group actions, there are evidences that 
the stationary measures seem not sufficient to describe the dynamics. On the other hand, 
the existence of some distinct concepts of topological pressure for group and semigroup actions 
makes it necessary to test their effectiveness to describe the dynamics.   
In the case of free semigroup actions, the coding of the dynamics by the full shift suggests to
consider the skew-product
\begin{equation}\label{de.Skew product}
\begin{array}{rccc}
F : & \{1,2,\dots, k\}^{\mathbb N}  \times X & \to & \{1,2,\dots, k\}^{\mathbb N}  \times X \\
	& (\omega,x) & \mapsto & (\sigma(\omega), g_{\omega_1}(x)).
\end{array}
\end{equation}
Moreover, a random walk on the semigroup $\mathcal S$ can be modelled by a 
Bernoulli probability measure $\mathbb P$ on $\{1,2,\dots, k\}^{\mathbb N}$. The 
pressure $P_{\text{top}}(\mathbb S, \phi, \mathbb P)$ of the semigroup action determined by that random walk
coincides with the annealed topological pressure $P^{(a)}_{\text{top}}(F, \tilde \phi, \mathbb P)$ 
of the random dynamical system determined by $F$, associated to the potential 
$\tilde \phi : \{1,2, \dots, k\}^{\mathbb N} \times X \to \mathbb R$ given by $\tilde \phi(\om,x)=\phi(x)$
(cf. Proposition~4.1 in \cite{CRV18}).
In particular, $P_{\text{top}}(\mathbb S, \phi, \mathbb P)$ coincides with the logarithm of the
spectral radius of the averaged transfer operator  
$$
\cA_1(f)  = \int L_{g_{\om}}(f)\, d\mathbb P(\omega).
$$
Furthermore, if $P_{top}(\mathbb S,0,\mathbb P)<\infty$ then entropy and invariant measures
can be defined through a functional analytic approach, which culminates in the variational principle
\begin{equation}\label{VPrinciple}
P_{\text{top}}(\mathbb{S},\phi,\mathbb P) = \sup_{\{\nu \, \in \, \mathcal{M}(X)\,\colon\, \Pi(\nu,\sigma) \neq \emptyset\}} \,\,\Big\{h_\nu(\mathbb{S},\mathbb P) + \int \phi \, d\nu\Big\}
\end{equation}
(we refer the reader to \cite{CRV18} for the definitions and more details). 
If all generators are Ruelle-expanding maps and $\phi$ is H\"older continuous then there exists a unique equilibrium state for the semigroup action $\mathbb S$ with respect
to $\phi$ and this can be characterized either as a marginal of the unique equilibrium state
for the annealed random dynamics or as the unique probability on $X$
obtained as limit of the equidistribution along pre-orbits associated to the 
semigroup dynamics by
$$
e^{-n P_{\text{top}}(\mathbb S, \phi, \mathbb P)} \cA_1^{*n} \delta_x
	= e^{-n P_{\text{top}}(\mathbb S, \phi, \mathbb P)} 
		\int_{\cW_n} \; \Big[ \sum_{g_\om(y)=x} \delta_y \, \Big] d\mathbb P(\omega)
$$
 (we refer the reader to \cite[Section~9]{CRV17} and \cite[Theorem~B]{CRV18} for more details).
A more general formulation, considering more general probabilities on semigroup actions rather than
random walks, was not available up to now as the thermodynamic formalism of the associated annealed dynamics needed to be described through a sequence of transfer operators instead of a single averaged operator.

\medskip
Our results allow not only to consider the thermodynamic formalism of semigroup actions with respect to 
more general probabilities in the base, but also to provide important asymptotic information on the convergence
to equilibrium states. Indeed, in general if one endows the semigroup $\mathcal S$ with a probability generated by a Markov measure
$\mathbb P$ on $ \{1,2, \dots, k\}^{\mathbb N}$ then it is natural to define the topological pressure 
of the semigroup action $\mathbb S$ by  
\begin{equation}\label{def:PP}
P_{\text{top}}(\mathbb S, \phi, \mathbb P)
	=\limsup_{n\to\infty} \frac1n \log \|\cA_n(1)\|_\infty
\end{equation}
where, as before,
$
\cA_n(f) = \int_{\om \in \cW_n} L_{g_{\om_1 \om_2 \dots \om_n}}(f)\, d\mathbb P(\om) 
$
(compare to the definition of topological pressure of a semigroup action in  \cite[Subsection~2.6]{CRV18}). Our main results have the following immediate consequences.
  
  \begin{corollary}\label{thm:VP}
Given $x\in X$, the sequence of probability measures on $X$ defined as
 $$
\nu^x_n:=\frac{\mathcal{A}^*_n(\delta_x)}{\mathcal{A}_n(\mathbf{1})(x)}, \quad n\ge 1  
 $$
 is weak$^*$ convergent to some probability $\nu=h d\pi$ on $X$  (independently of $x$).
  Moreover, the convergence is exponentially fast with respect to the Vaserstein distance.
  \end{corollary}

\subsection{A boundary of equilibria} As in the section before, we now assume that $X$ is compact and that there is only one potential $\varphi: X \to \mathbb{R}$. However, in contrast to the approach via the free semigroup, we are now interested in identifying elements in the semigroup $\mathcal{S}$ which are dynamically close and use this information in order to define a compactification of the discrete set $\mathcal{S}$.
However, as the topology will rely on the associated equilibrium states, we have to extend the semigroup by considering also the potential function. That is, for $\mathbb{G}_1 := \{ (g_1,\varphi),(g_2,\varphi),\ldots (g_k,\varphi)\}$, we consider 
$$ \mathbb{G}  := \left\{ (g,\psi)  : \exists n \in \mathbb{N}, j_1, \ldots, j_n \hbox{ s.t. } (g,\psi) = (g_{i_1},\varphi) \ast     \cdots \ast (g_{i_n}, \varphi)  \right\},  
$$
where $$(g_1,\psi_1) \ast  (g_2,\psi_2) := (g_1 \circ g_2,\psi_2 +\psi_1\circ g_2)$$
is also the product on $\mathbb G$.

As a first step, we begin with the definition of a metric on the countable set $\cW^{\ast} := \{w : |w|< \infty\}$ of finite words. For finite words $v=(v_1 \ldots v_m)$  and $w=(w_1 \ldots w_n)$ in $\cW^{\ast}$, set $d_{\cW^{\ast}}(v,w) =0$ for $v=w$ and  
\begin{align*} 
d_{\cW^{\ast}}(v,w) :=&  2^{-\min\{ k:  v_k \neq w_k \hbox{ or } k > \min\{m,n\} \} }\\ &  + 2^{-\min\{ k:  v_{m+1-k} \neq w_{n+1-k} \hbox{ or } k > \min\{m,n\} \} }, \end{align*}
for $v \neq w$. Observe that $d_{\cW^{\ast}}$ is a metric, that $\cW^{\ast}$ is discrete with respect to this metric and that two words are close if they have the same beginning and ending. In particular, Cauchy sequences either have to be eventually constant or have to grow from the interior of a word. The reason for this construction is based on the following observation. 
Let $\underline{w}$ and $\overline{w}$ refer to the periodic extensions of $w$ to the left and the right, respectively, as defined in Remark \ref{remark-adic flow}. Then, by Proposition \ref{prop:equilibrium}, the map $w \to \mu_{\underline{w},\overline{w}}$ is H\"older continuous with respect to $d_{\cW^{\ast}}$. In particular, $d_{\cW^{\ast}}$ can be seen as a metric on the free semigroup which is compatible with the Vaserstein distance of the associated equilibrium states.  

Secondly, we define a metric on $\mathbb{G}$ which does not depend on the choice of $w \in {\cW^{\ast}}$ for the representation of  $(g,\psi)= (T_w, \varphi_w)$.
In order to do so, define, for $g \in \mathcal{S}$, 
\[ \kappa(g) : = \lim_{\epsilon \to 0} \inf\left\{ \frac{d(g(x),g(y))}{d(x,y)} :0<  d(x,y) < \epsilon \right\},\]  
and note that, as the semigroup is Ruelle expanding with parameter $\lambda \in (0,1)$, we have that $\kappa(T_w) \geq \lambda^{-|w|}$. Furthermore, for $(g,\psi) \in \mathbb{G}$, let $\mu_{g,\psi}$ be the unique equilibrium state 
for the potential $\psi$ and the map $g$, that is, if  $(g, \psi)= (T_w, \varphi_w)$, then 
$\mu_{g,\psi} = \mu_{\underline{w},\overline{w}}$.  
Now set
\[ d_{\mathbb{G}}((g,\psi_1),(h,\psi_2)) := 
\begin{cases} \overline{W}(\mu_{g,\psi_1},\mu_{h,\psi_2})  + \frac{1}{\kappa(g)} + \frac{1}{\kappa(h)}
&:\;   (g,\psi_1)\neq (h,\psi_2)\\
0 &:\; (g,\psi_1)=(h,\psi_2).
\end{cases}
 \]   
The following proposition summarises the basic topological facts. The proof is omitted as the assertions almost immediately follow from the definitions and Proposition \ref{prop:equilibrium}.
\begin{proposition} \label{prop:boundary}
Assume that $g_1, \ldots, g_k$ are Ruelle expanding and jointly topological mixing, and that $\varphi$ is H\"older continuous. Then, for the objects defined above, the following holds. 
\begin{enumerate}  
\item $(\cW^{\ast}, d_{\cW^{\ast}})$  and  $({\mathbb{G}}, d_{\mathbb{G}})$ are discrete, metric spaces. \item The map $w \mapsto (T_w,\varphi_w)$ is H\"older continuous.
\item A sequence $((g_n,\psi_n))_n$ in $\mathbb{G}$ is a Cauchy sequence if and only if $\kappa(g_n) \to \infty$ and  $(\mu_{g_n,\psi_n})$  converges in the weak$^\ast$-topology. Moreover, two Cauchy sequences have the same limit if and only if their sequences of equilibrium states have the same limit. 
\item For the boundary $\partial \mathbb{G}$ of the completion with respect to $d_{\mathbb{G}}$, identified with limits of Cauchy sequences $((g_n,\psi_n))_n$ in $\mathbb{G}$, we have that the map 
\[ \partial \mathbb{G} \to \{ \mu_{\sigma,\omega} : \sigma \in \Sigma^-, \omega \in \Sigma\}, \; (({g_n,\psi_n}))_n \mapsto \lim_{n \to \infty} \mu_{g_n,\psi_n}  \]
is Lipschitz continuous and onto.  
\end{enumerate}
\end{proposition}

Observe that the result provides a description of  $\partial \mathbb{G}$ as a set of equivalence classes of Cauchy sequences, that is two sequences are considered to be equivalent if they have the same limit.
However, it seems to be impossible to obtain an explicit description of  $\partial \mathbb{G}$ in general. We close with two examples where this is possible. In the first example, $\partial \mathbb{G}$  
is trivial whereas in the second example, $\partial \mathbb{G}$ is equal to $\Sigma^{-}$.  

\begin{proposition} 
If $\mathbb{G}$ is Abelian, then $\partial \mathbb{G}$ is a point.
\end{proposition}
\begin{proof} Assume that $(g_1,\psi_1), (g_2,\psi_2) \in  \mathbb{G}$, and denote by $L_i$ the corresponding Ruelle operators. As $\mathbb{G}$ is Abelian, it immediately follows that $L_1L_2 =L_2L_1$. Now assume that the $h_i$ are the unique positive H\"older functions (up to colinearity) and $\lambda_i>0$ such that $L_i(h_i) = \lambda_i h_i$, given by Ruelle's theorem. Hence, $L_2(L_1 (h_2)) = L_1(L_2 (h_2)) =  \lambda_2 L_1(h_2)$. As $L_1(h_2)$ is positive, it follows that  $L_1(h_2)$ and $h_1$ are colinear, that is $L_1(h_2)$ is a multiple of $h_1$ and $\lambda_1 = \lambda_2$. The same argument then shows that the $L_i^\ast$-eigenmeasures coincide. Hence, after normalising, we obtain that $\mu_{g_1,\psi_1} = \mu_{g_2,\psi_2}$. In particular, $\{ \mu_{\sigma,\omega} : \sigma \in \Sigma^-, \omega \in \Sigma\}$ is a singleton.
\end{proof}

\begin{example} Let $T:[0,1] \to [0,1]$, $x \mapsto 4x (\mbox{mod} 1)$ and $S = U^{-1} T U$, where
\[ U:[0,1] \to [0,1], \qquad
x \mapsto 
\begin{cases}
 3x/2 & 0 \leq x < 1/8 \\
 x + 1/16 & 1/8 \leq x < 3/8 \\
 x/2 + 1/4 & 3/8 \leq x < 1/2 \\
 x & 1/2 < x \leq 1. 
\end{cases}  \]
The semigroup $\mathcal{S}$ generated by $\{S,T\}$ is a free semigroup, that is two elements in $\mathcal{S}$ coincide if and only if they have the same representation as a product of the generators. 
Moreover, $\partial \mathbb{G} \cong \Sigma^{-}$, where $\mathbb{G}$ is the semigroup generated by $(T,0)$ and  $(S,0)$.  
\end{example}

\begin{proof} The proof relies on the construction of a family of renormalization operators acting on the set of orientation-preserving homeomorphisms $f$ in such a way that 
$$ T^n \circ \Xi_n(f)  =  f \circ T^n,$$ 
as this allows to associate to each element 
$g=S^{m_k}T^{n_k} \cdots S^{m_1}T^{n_1}$ in $\mathcal{S}$ a uniquely determined normal form $T^{m_1 + n_1 + \cdots m_k+ n_k} \circ f_g$, where $f_g$ is an orientating preserving homeomorphism. The uniqueness of the normal form is a consequence of the choice of $U$ as the compositions with $U$ and $U^{-1}$ act as markers in the following way. For an orientating preserving homeomorphism $f$, it is shown below that  $\|\Xi^n(f) - \mathrm{id}\|_\infty =  4^{-n} \| f- \mathrm{id}\|_\infty$, and that the composition $\Xi_n(f)\circ U^{\pm 1}$ leaves invariant the right half of $\Xi_n(f)$ whereas the left half is marked by a positive or negative bump of size bigger than $\|\Xi^n(f) - \mathrm{id}\|_\infty$.

\subsubsection*{Construction and properties of $\Xi_n$.} Let $f: [0,1] \to [0,1]$ be a homeomorphism which fixes $0$ and $1$ and define, for $x \in [k/4^n, (k+1)/4^n]$,
\[ \Xi_n(f)(x) := \left( T^n|_{[k/4^n, (k+1)/4^n]} \right)^{-1} \circ f \circ T^n(x) = 4^{-n}(f(4^n x - k) + k). \]   
Then, as it can be easily seen, $ T^n \circ \Xi_n(f)  =    f \circ T^n $ and $\Xi_n(f)(k/4^n) = k/4^n$ for all $k= 0, \ldots, 4^n$. In particular, as  $\Xi_n(f)|_{[k/4^n, (k+1)/4^n]}$ is a homeomorphism, $\Xi_n(f)$ is a homeomorphism. Moreover, for $x \in [k/4^n, (k+1)/4^n]$, we have  
\begin{align*}\Xi_n(f)(x) - x & = 4^{-n}(f(4^n x - k) + k) - x \\
&  = 4^{-n}(f(4^n x - k)  - (4^{n}x-k)) 
 =   4^{-n}(f\circ T^n(x) - T^n(x) ).    
\end{align*}
That is, $\Xi_n$ contracts the distance to the identity by the factor $4^{-n}$. We now proceed with an analysis of the concatenations $\Xi_n(f)\circ U$ and  $\Xi_n(f)\circ U^{-1}$, where $f$ is a homeomorphism  with $\|f - \textrm{id} \|_\infty \leq 1/12$. First note that 
\[ U(x) - x =\begin{cases}
 x/2 & x \in [0,\frac{1}{8}) \\
 1/16 & x \in [\frac{1}{8},\frac{3}{8}) \\
 -x/2 + 1/4 &  x \in [\frac{3}{8},\frac{1}{2}) \\
 0 &  x \in [\frac{1}{2},1] 
\end{cases},  \quad U^{-1}(x) - x =\begin{cases}
 - x/3 & x \in [0,\frac{3}{16}) \\
 - 1/16 & x \in [\frac{3}{16},\frac{7}{16}) \\
 x  - 1/2 & x \in [\frac{7}{16},\frac{1}{2})\\
 0 & x \in [\frac{1}{2},1]  
\end{cases}  \]
and observe that, by construction, $\Xi_n(f)  - \textrm{id}$ is periodic with period $4^{-n}$. However, as $[\frac{1}{8},\frac{3}{8})$, $ [\frac{3}{16},\frac{7}{16})$ and $[\frac{1}{2},1]$  are all of length bigger than or equal to $1/4$, we obtain that
\begin{align*}
\nonumber	\max_{x \in [0,1]} \left(\Xi_n(f)(U(x)) - x \right) 
& = \max_{x \in [\frac{1}{8},\frac{3}{8})} \left(\Xi_n(f)(U(x)) - U(x) + U(x)-x \right) \\
\label{eq:interval of max}
& = 4^{-n} \max_{x \in [0,1]} \left(f(x) - x \right) + \frac{1}{16} = \frac{1}{4^n \cdot 12}  + \frac{1}{16} \leq \frac{1}{12}, 
\end{align*}
and, repeating the argument, $\|\Xi_n(f)\circ U^{j} - \textrm{id}\|_\infty \leq 1/12$, for $j = \pm 1$. 

In other words, the space $\mathfrak{H}$ of orientation-preserving homeomorphisms with $\|f - \textrm{id}\|_\infty \leq 1/12$ is invariant under the operation $f \mapsto \Xi_n(f)\circ U^{j}$. Moreover, we have that 
\begin{equation} \label{eq:marker} \|\Xi_n(f)\circ U^j - U^j\|_\infty  = 4^{-n} \left\| \Xi_n(f)  -  \textrm{id} \right\|_\infty  =  4^{-n} \| f - \textrm{id} \|_\infty \leq \frac{1}{48}. \end{equation}

\medskip
\subsubsection*{Coding of $\mathbb{G}$.} Assume that $g = S^{m_k}T^{n_k} \cdots S^{m_1}T^{n_1}$
 for some $k \in \mathbb{N}$ and $m_i,n_i \in \mathbb{N} \cup \{0\}$. As 
$U, U^{-1} \in \mathfrak{H}$,
  it follows from an iterated application of $ \Xi_n(\cdot)\circ U^{j}$ that there exists a  homeomorphism $f_g \in  \mathfrak{H}$ such that $g = T^n \circ f_g$, where $n = \sum_{i=1}^k m_i + n_i$. Moreover, as $T^n$ is a local homeomorphism, $f = f_g$ is uniquely determined. 

Now assume that $g = S^{m_k}T^{n_k} \cdots S^{m_1}T^{n_1} \in \mathcal{S}$ where, without loss of generality, $m_1,\ldots,m_{k-1} \neq 0$ and $n_2,\ldots,n_{k} \neq 0$. We now show how to determine $m_1$ and $n_1$ from $f$ in a unique way.
\begin{description}[font=\normalfont]
 \item[\textit{Case 1}] If $m_1 = 0$, then $k=1$,  $g=T^{n_1}$ and $f = \textrm{id}$. 
 \item[\textit{Case 2}] If $m_1 \neq 0$ and $n_1 \neq 0$, then $k > 1$ and, for $\bar{f}:= f_{S^{m_k}T^{n_k} \cdots S^{m_1}}$, we have that $f = \Xi_{n_1}(\bar{f})$. It now follows from \eqref{eq:marker} that  $\bar{f} - \textrm{id}$ is strictly positive on $[ {1}/{8},{3}/{8}]$ and has zeros in $[1/2,1]$. Therefore $n_1$ is determined by the periodicity of $f - \textrm{id}$, and $\bar{f}(x) = f(2^{n_1})(x)$. The value of $m_1$ is then determined by applying Case 3 to $S^{m_k}T^{n_k} \cdots S^{m_1}$ and $\bar{f}$.
 \item[\textit{Case 3}]   
If $m_1 \neq 0$ and $n_1 =0$, then $k \geq 1$ and, 
for $\bar{f}:= f_{S^{m_k}T^{n_k} \cdots T^{m_2}}$, we have that $f = \Xi_{m_1}(\bar{f} \circ U^{-1})\circ U$ or, equivalently,  $f\circ U^{-1} = \Xi_{m_1}(\bar{f})$. Hence, in order to repeat the above argument based on periodicity, we have to show that  the left half of $\bar{f} - \textrm{id}$ is somehow marked. If $k=1$, then $\bar{f} = U^{-1}$ and, in particular,  $\bar{f}$
 is strictly negative on  $[{3}/{16}, {7}/{16}]$ and has zeros in $[1/2,1]$. Hence, $m_1$ can be determined through the period of $f\circ U^{-1}$. On the other hand, if $k>1$ then $n_2> 0$ and the same argument is applicable as   \eqref{eq:marker} implies that $\bar{f}$
 is strictly negative on  $[{3}/{16}, {7}/{16}]$ and has zeros in $[1/2,1]$.       
\end{description} 
By iterating this procedure, one then recovers $m_2, \ldots, m_k$ and $n_2, \ldots, n_k$ from $f$. Furthermore, as the $m_i$ and $n_i$ only depend on the period, it follows that the relation between $f$ and these values is one-to-one. 
This then implies that the map 
\[ \mathcal S
\to \{f_g : g \in \mathcal{S}\}, \quad 
(w_1 \ldots w_n) \mapsto  f_{w_n \circ \cdots \circ w_1}   \]
is a bijection, and, as an immediate corollary, $\mathcal{S}$ is a free semigroup.

\subsubsection*{The associated measures of maximal entropy.}  
Now fix a H\"older function $h$, an element $g \in \mathcal{S}$ and let $n\in \mathbb{N}$ be given by $g = T^n\circ  f_g$. Then the Ruelle operators $L_g$ and $L_T$ associated to $g$ and $T$, respectively, satisfy 
\begin{align*} L_g(h)(x) & = \sum_{g(y)=x} h(y)  =   \sum_{T^n z=x} h(f_g^{-1}(z)) =  L_T^n(h\circ f_g^{-1})(x),\\
\frac{L_g(h L_g(\mathbf{1}))}{L_{g^2}(\mathbf{1})}  & = \frac{L_g(4^n h )}{4^{2n}} = \frac{1}{4^n}  L_T^n(h\circ f^{-1}).
\end{align*}  
By Proposition \ref{prop:equilibrium}, the measures of maximal entropy $\mu_g$ and $\mu_T$ of $g$ and $T$, respectively, satisfy $\overline{W}(\mu_g,\mu_T\circ f_g) \ll s^n$. Hence, $\mu_g = \lim_{l \to \infty} \mu_T\circ f_{g^l}$. However, this result also implies that,
for an infinite word $ (v_i) \in \{S,T\}^{\mathbb N}$,  the sequence $ \mu_{g_{v_l \cdots v_1}}$ is a Cauchy sequence and therefore convergent. It remains to show that the mapping from $ (v_i) $ to this limit is injective. In order to do so, let $ (v_i) \neq (w_i) $ be different elements in $\{S,T\}^{\mathbb N}$. Then, by applying the construction of the $n_i$ and $m_i$ above to infinite words, it follows that $\mu_{g_{v_l \cdots v_1}} \neq   \mu_{g_{w_l \cdots w_1}}$ for all $l$ sufficiently large. Furthermore, it can be deduced from the recursive construction of $f_g$ that there exists an open set 
$A$ and $\epsilon > 0$ such that $f_{v_l \cdots v_1}(x) - f_{w_l \cdots w_1}(x) > \epsilon $ for all $x \in A$ and all $l$ sufficiently large. Hence, $\lim_l \mu_{g_{v_l \cdots v_1}} \neq  \lim_l \mu_{g_{w_l \cdots w_1}}$.
\end{proof}

\end{document}